\documentclass[smallextended,envcountsect]{svjour3}
\smartqed
\usepackage{graphicx}
\journalname{}

\usepackage{amsmath}
\usepackage{amssymb}
\usepackage{algpseudocode}
\usepackage{algorithm}
\usepackage{mathabx}

\renewcommand{\H}{\mathcal{H}}
\newcommand{\R}{\mathbb{R}}
\newcommand{\G}{\mathcal{G}}

\newcommand{\C}{\mathcal{C}}

\newcommand{\ran}{\ensuremath{\operatorname{ran}}}

\newcommand{\dom}{\operatorname{dom}}
\newcommand{\prox}{\operatorname{prox}}

\newcommand{\sgn}{\operatorname{sgn}}
\newcommand{\Id}{\operatorname{Id}}

\begin{document}

\title{A Projected Variable Smoothing for Weakly Convex Optimization and Supremum Functions}
\titlerunning{A Projected Variable Smoothing for Weakly Convex Optimization}

\author{Sergio L\'opez-Rivera, Pedro P\'erez-Aros and Emilio Vilches}

\institute{Sergio L\'opez-Rivera \at 
            Departamento de Ingenier\'ia Matem\'atica, Universidad de Chile,
             Santiago,  Chile. \\ Email: sergio.lopez@dim.uchile.cl
            \and
            Pedro P\'erez-Aros \at
             Departamento de Ingenier\'ia Matem\'atica, Universidad de Chile,
             Santiago,  Chile\\
            Email: pperez@dim.uchile.cl
           \and
            Emilio Vilches \at Instituto de Ciencias de la Ingenier\'ia, 
               Universidad de O'Higgins, 
               Rancagua, Chile\\
               Email: emilio.vilches@uoh.cl
}

\date{Received: date / Accepted: date}

\maketitle

\begin{abstract}
In this paper, we address two main topics. First, we study the problem of minimizing the sum of a smooth function and the composition of a weakly convex function with a linear operator on a closed vector subspace. For this problem, we propose a projected variable smoothing algorithm and establish a complexity bound of $\mathcal{O}(\epsilon^{-3})$ to achieve an $\epsilon$-approximate solution.  Second, we investigate the Moreau envelope and the proximity operator of functions defined as the supremum of weakly convex functions, and we compute the proximity operator in two important cases. In addition, we apply the proposed algorithm for solving a distributionally robust optimization problem, the LASSO with linear constraints, and the max dispersion problem. We illustrate numerical results for the  max dispersion problem.
\end{abstract}

\section{Introduction}
 In this paper, we consider the following problem 
\begin{equation}\label{2.1}
\min_{x\in V}h(x)+g(Ax),
\end{equation}
where $V\subset\H$ is a closed vector subspace of a real Hilbert space $\H$, $h\colon\H\rightarrow\R$ is a differentiable function with a Lipschitz gradient, $g$ is weakly convex and lower semicontinuous, and $A$ is a nonzero bounded linear operator from $\H$ to a real Hilbert space $\G$. Whenever $V=\H$, the problem has been studied in \cite{bohm2021variable}. When  $g$ is convex, it has been analyzed in \cite{MR4534446}. Additionally, if  $A=\Id$, the problem can be solved using the algorithm proposed in \cite{briceno2015forward}. If $V=\H$, we can also apply the method proposed in \cite{condat2013primal} in the convex case. The class of weakly convex functions is a generalization of the class of convex functions and there are several classes of functions that are weakly convex. For example, every continuously differentiable piecewise quadratic function defined from $\R^{n}$ to $\R$ is weakly convex. Another  example is the composition of a convex Lipschitz continuous function and a smooth function with Lipschitz gradient. We refer to \cite{MR4359983} for more details about this important class of functions.

One of the problems of the form \eqref{2.1} is the constrained LASSO problem \cite{tibshirani2011solution} that combines sparse minimization and least squares under linear constraints. Also, problem \eqref{2.1} appears in PDEs \cite[Section~3]{MR570823}, signal and image processing \cite{Aujol2006,MR1440119}, stochastic traffic theory \cite{MR3703499}, among other fields. In the previous applications, the vector subspace constraint
models intrinsic properties of the solution as non-anticipativity in stochastic problems.

 We are interested in the case when the function $g$ is a supremum function. The supremum function arises in several optimization problems \cite{MR4361979,sun2021robust,MR4362585}. An important problem that involves the supremum function is the called \textit{Distributionally Robust Optimization} (DRO) problem, which assumes that the true probability distribution belongs to a set of distributions called the \textit{ambiguity set}. Optimization is then performed based on the worst-case scenario within this ambiguity set. The DRO model covers the Robust Optimization problem when the ambiguity set includes all possible distributions, and Stochastic Optimization when the ambiguity set consists of a single probability distribution. Another application which appears a weakly convex supremum function is the \textit{max dispersion problem} (or max-min location problem) \cite{jeyakumar2018exact}. 
 
The proposed algorithm for solving \eqref{2.1} is a generalization of the algorithm proposed in \cite{bohm2021variable} and includes a projection onto the vector subspace. In the proposed method we need to compute the proximity operator of the function $g$. To do this, we generalize \cite[Theorem~3.5]{MR4279933}, which allows calculating the proximity operator of the supremum of convex functions.

The remainder of the paper is organized as follows. In Sect.~\ref{sec_pre}, we introduce the notation and the basic definitions. Sections~\ref{sec_alg} and \ref{sec_epoch} describes the proposed algorithm and its convergence properties. As in \cite{bohm2021variable}, the convergence of the algorithm is obtained under the hypothesis of that $g$ is Lipschitz. We prove that if the sequence generated by the algorithm is bounded, then we can eliminate such hypothesis to obtain convergence. In Sect.~\ref{sec_mor_sup}, we prove that, in the context of weakly convex functions, the Moreau envelope of a supremum function is the supremum of Moreau envelopes. As a consequence of this result, we investigate the proximity operator of the supremum of weakly convex functions, and we compute the proximity operator in two important cases. In Sect.~\ref{sec_aplic}, we use the results proved in the previous sections for solving several applications including the DRO problem, the LASSO with linear constraints, and the max dispersion problem. We illustrate numerical results for the max dispersion problem in Section~\ref{sec_num}. Finally, we mention some conclusions in Section~\ref{sec_conc}.
\section{Preliminaries}
\label{sec_pre}
Let $\H$ be a real Hilbert space with inner product $\langle\cdot,\cdot\rangle$ and induced norm $\|\cdot\|$. We denote by $\Gamma_{0}(\H)$ the set of extended valued functions oh $\H$ that are proper, convex, and lower semicontinuous. Moreover, for $\rho\geq 0$, we denote by $\Gamma_{\rho}(\H)$ the set of functions $g$ such that $g+\frac{\rho}{2}\|\cdot\|^{2}\in\Gamma_{0}(\H)$.  We say that a function is \emph{weakly convex} if the function belongs to $\Gamma_{\rho}(\H)$ for some $\rho\geq 0$. For $\mu>0$, the Moreau envelope of an extended valued function $f$ is defined by
$$
f_{\mu}(x):=\inf_{y\in\H}\,f(y)+\frac{1}{2\mu}\|x-y\|^{2}.
$$
If $f\in\Gamma_{\rho}(\H)$ and $\mu<1/\rho$, the above infimum is attained at a unique point and defines the so-called \emph{proximity operator}  $x\mapsto \prox_{\mu f}(x)$ (see, e.g., \cite[Theorem~3.4]{MR1230710}). For a closed set $\mathcal{Q}\subset \H$, the distance function and the projection mapping are defined as follows:
$$
d(x,\mathcal{Q}):=\inf_{y\in \mathcal{Q}}\Vert x-y\Vert \textrm{ and } P_{\mathcal{Q}}(x)=\operatorname{argmin}_{y\in\mathcal{Q}}\|x-y\|.
$$
The subdifferential of a weakly convex function $f$  (see, e.g., \cite{MR1230710}) is defined by
$$
\partial f(x)=\partial( f+\frac{\rho}{2}\Vert \cdot\Vert^2)(x)-\rho x \textrm{ for } x\in \H,
$$
where $\rho\geq 0$ is chosen such that $f+\frac{\rho}{2}\in \Gamma_0(\H)$, and the subdifferential on the right-hand side is understood in the sense of convex analysis. 

\noindent In this paper, we consider the problem 
\begin{align}
\label{prob_w_conv}
\min_{x\in V}h(x)+g(Ax),
\end{align}
where $V\subset\H$ is a closed vector subspace, $h\colon\H\rightarrow\R$ is a differentiable function with a Lipschitz gradient of constant $L_{\nabla h}$, $g\in \Gamma_{\rho}(\H)$ for some $\rho\geq 0$, and $A$ is a nonzero bounded linear operator from $\H$ to a real Hilbert space $\G$. We say that $x^{*}\in\H$ is a critical point of \eqref{prob_w_conv} if
$$
0\in \nabla h(x^{*})+A^{*}\partial g(Ax^{*})+N_{V}(x^{*}).
$$
In the line of \cite{bohm2021variable}, for every $k\geq 1$, we consider the following regularized problem
\begin{align}
\label{prob_w_conv_reg}
\min_{x\in V}F_{k}(x):=h(x)+g_{\mu_{k}}(Ax),
\end{align}
where $\mu_{k}>0$ and $2\rho\mu_{k}\leq 1$. By \cite[Lemma~3.1]{bohm2021variable},  $\nabla F_{k}$ is Lipschitz with constant 
\begin{align}
\label{cons_L_k}
L_{k}:=L_{\nabla h}+\dfrac{\|A\|^{2}}{\mu_{k}},
\end{align}
where $L_{\nabla h}$ is the Lipschitz constant of $\nabla h$.
Note that $x$ is a critical point of \eqref{prob_w_conv_reg} if and only if 
\begin{align*}
0\in \nabla F_{k}(x)+N_{V}(x) \, \Leftrightarrow\, x=P_{V}(x-\gamma\nabla F_{k}(x)) \textrm{ for all } \gamma>0,
\end{align*}
which motivates us to use the algorithm: $x_{k+1}=P_{V}(x_{k}-\gamma_{k}\nabla F_{k}(x_{k}))$.

  \section{Algorithm Projected Variable Smoothing}
\label{sec_alg}
We propose the following algorithm.
\begin{algorithm}[H]
\caption{Projected Variable Smoothing}\label{alg_w_convex_sev}
\begin{algorithmic}
\Require $x_{1}\in V$, $\alpha\in\left]0,1\right[$, and $C>0$ such that $2\rho C\leq 1$.
\For{$k=1,2,\ldots$}
\State Set $\mu_{k}\gets Ck^{-\alpha}$, set $L_{k}$ as in \eqref{cons_L_k}, set $\gamma_{k}\gets 1/L_{k}$.
\State $x_{k+1}\gets P_{V}(x_{k}-\gamma_{k}\nabla F_{k}(x_{k}))$.
\EndFor
\end{algorithmic}
\end{algorithm}
In what follows, we assume that the sequence $(F_{k}(x_{k}))$ is bounded from below by $F^{*}$.  We have the following convergence result, which, for $\alpha=1/3$, $C=(2\rho)^{-1}$, and $V=\H$ recovers \cite[Theorem~4.1]{bohm2021variable}. It is worth emphasizing that by controlling $\alpha$, we can balance the trade-off between satisfying the optimality condition and the convergence of the algorithm.
\begin{theorem}
\label{teo_conv_weak_sev}
In the context of problem~\eqref{prob_w_conv}, assume that $g$ is $L_{g}$-Lipschitz. Then, the sequence defined by Algorithm~\ref{alg_w_convex_sev} satisfies
\begin{equation*}
\begin{aligned}
\min_{1\leq j\leq k}d(-\nabla h(x_{j}), A^{*}\partial g(\prox_{\mu_{j}g}(Ax_{j}))+N_{V}(x_{j}))&\leq k^{\frac{\alpha-1}{2}}\widetilde{C} \,\textrm{ for all } k\geq 1, \\
\|Ax_{k}-\prox_{\mu_{k}g}(Ax_{k})\| &\leq k^{-\alpha}C L_{g} \,\textrm{ for all } k\geq 1,
\end{aligned}
\end{equation*}
where $\widetilde{C}:=\dfrac{\sqrt{2}\sqrt{L_{\nabla h}+C^{-1}\|A\|^{2}}}{\sqrt{2^{1-\alpha}-1}}\sqrt{F_{1}(x_{1})-F^{*}+C L_{g}^{2}}$.
\end{theorem}
\begin{proof}
By the descent lemma (see, e.g., \cite[Lemma~5.7]{MR3719240}),  for all $k\geq 1$
\begin{align*}
F_{k}(x_{k+1})\leq F_{k}(x_{k})+\langle \nabla F_{k}(x_{k}),x_{k+1}-x_{k}\rangle+\dfrac{1}{2\gamma_{k}}\|x_{k+1}-x_{k}\|^{2}.
\end{align*}
Using that $P_{V}$ is linear, $x_{k}\in V$, and  that $\langle x,P_{V}y\rangle=\langle P_{V}x,P_{V}y\rangle$, we have
$$
\langle \nabla F_{k}(x_{k}),x_{k+1}-x_{k}\rangle=\langle\nabla F_{k}(x_{k}),-P_{V}\gamma_{k}\nabla F_{k}(x_{k})\rangle
=-\gamma_{k}\|P_{V}\nabla F_{k}(x_{k})\|^{2}.
$$
 Thus, since $\|x_{k+1}-x_{k}\|^{2}=\gamma_{k}^{2}\|P_{V}\nabla F_{k}(x_{k})\|^{2}$, we get that
\begin{align}
\label{des_Fk}
F_{k}(x_{k+1})\leq F_{k}(x_{k})-\dfrac{\gamma_{k}}{2}\|P_{V}\nabla F_{k}(x_{k})\|^{2}.
\end{align}
By \cite[Lemma~4.1]{bohm2021variable}, we have that
$$  \hspace{-0.01mm}
F_{k+1}(x_{k+1})\leq F_{k}(x_{k+1})+\dfrac{1}{2}(\mu_{k}-\mu_{k+1})\dfrac{\mu_{k}}{\mu_{k+1}}L_{g}^{2}\leq F_{k}(x_{k+1})+(\mu_{k}-\mu_{k+1})L_{g}^{2},
$$
where we have used that $\mu_{k}\leq 2^{\alpha}\mu_{k+1}\leq 2\mu_{k+1}$. Hence, we obtain that
\begin{equation*}
F_{k+1}(x_{k+1})\leq F_{k}(x_{k})-\dfrac{\gamma_{k}}{2}\|P_{V}\nabla F_{k}(x_{k})\|^{2}+(\mu_{k}-\mu_{k+1})L_{g}^{2}.
\end{equation*}
By summing these inequalities from $k=1$ to $K$, we obtain
\begin{equation}\label{des_pv_grad}
\begin{aligned}
\sum_{k=1}^{K}\dfrac{\gamma_{k}}{2}\|P_{V}\nabla F_{k}(x_{k})\|^{2}&\leq F_{1}(x_{1})-F_{K+1}(x_{K+1})+(\mu_{1}-\mu_{K+1})L_{g}^{2}\\
&\leq F_{1}(x_{1})-F^{*}+\mu_{1}L_{g}^{2}.
\end{aligned}
\end{equation}
On the one hand, we observe that 
\begin{align}
\label{des_gammak}
\gamma_{k}=\frac{\mu_{k}}{\mu_{k}L_{\nabla h}+\|A\|^{2}}\geq k^{-\alpha}\frac{C}{C L_{\nabla h}+\|A\|^{2}}=k^{-\alpha}\dfrac{1}{L_{\nabla h}+C^{-1}\|A\|^{2}}.
\end{align}
On the other hand,  
\begin{align*}
\sum_{k=1}^{K}k^{-\alpha}\geq \int_{1}^{K+1}x^{-\alpha}dx\geq (K+1)^{1-\alpha}-1\geq (2^{1-\alpha}-1)K^{1-\alpha},
\end{align*}
where we have used Lemma \ref{claim_des}. Hence,
\begin{align}
\label{des_min_proj_grad}
\min_{1\leq j\leq K}\|P_{V}\nabla F_{j}(x_{j})\|^{2}\leq \dfrac{2(L_{\nabla h}+C^{-1}\|A\|^{2})}{(2^{1-\alpha}-1)K^{1-\alpha}}\left(F_{1}(x_{1})-F^{*}+C L_{g}^{2}\right).
\end{align}
Now, $N_{V}(x)=V^{\perp}$ for all $x\in V$. Then, for all $z\in\H$ and $x\in V$, we have that
\begin{align}
\label{des_cono_proj}
d(-z,N_{V}(x))=d(-z,V^{\perp})\leq \|-z-(-z+P_{V}z)\|=\|P_{V}z\|.
\end{align}
By \cite[Lemma~3.2]{bohm2021variable}, we have that 
$$\nabla(g_{\mu_{j}}\circ A)(x_{j})=A^{*}\nabla g_{\mu_{j}}(Ax_{j})\in A^{*}\partial g(\prox_{\mu_{j} g}(Ax_{j})).$$
 Thus,  taking $z=\nabla F_{j}(x_{j})=\nabla h(x_{j})+\nabla (g_{\mu_{j}}\circ A)(x_{j})$ in \eqref{des_cono_proj}, we obtain that
\begin{align}
\label{des_key_1} 
d(-\nabla h(x_{j}),A^{*}\partial g(\prox_{\mu_{j} g}(Ax_{j}))+N_{V}(x_{j}))\leq \|P_{V}\nabla F_{j}(x_{j})\|,
\end{align}
which, by virtue of  \eqref{des_min_proj_grad},  implies the first inequality. The second inequality follows from \cite[Lemma~3.3]{bohm2021variable}. \qed
\end{proof}
\begin{remark}
The proposed algorithm can be easily adapted to the case of a closed affine subspace. Indeed, if $V$ is a closed affine subspace, that is, $V=W+z_{0}$, where $W$ is a closed vector subspace and $z_{0}\in\H$, then problem~\eqref{prob_w_conv} is equivalent to $\min_{x\in W}h(x+z_{0})+g(y+Az_{0})$. The assumptions on the data hold, and $\prox_{\mu g(\cdot+ Az_0)}(y)=\prox_{\mu g}(y+Az_{0})-Az_{0}$. 
\end{remark}
The hypothesis that the function $g$ is Lipschitz in Theorem~\ref{teo_conv_weak_sev} may be too restrictive. For example, if $g$ is of the form $g(x)=-\|x\|^{2}$, then such a hypothesis does not hold. In what follows, we prove that if the sequence generated by the Algorithm~\ref{alg_w_convex_sev} is bounded and $\G=\R^{m}$, then a similar result to the Theorem~\ref{teo_conv_weak_sev} holds without requiring $g$ to be Lipschitz. 
\begin{theorem}
In the context of problem \eqref{prob_w_conv}, assume that $\G=\R^{m}$, and the sequence $(x_{k})$ generated by the Algorithm~\ref{alg_w_convex_sev} is bounded. Then, there exist $\ell\geq 0$ and $N\geq 1$ such that 
\begin{equation*}
\begin{aligned}
\min_{N\leq j\leq k}d(-\nabla h(x_{j}), A^{*}\partial g(\prox_{\mu_{j}g}(Ax_{j}))+N_{V}(x_{j}))&\leq k^{\frac{\alpha-1}{2}}\widetilde{C} \textrm{ for all } k\geq N,\\
\Vert Ax_{k}-\prox_{\mu_{k}g}(Ax_{k})\Vert &\leq k^{-\alpha}C \ell  \textrm{ for all } k\geq N,
\end{aligned}
\end{equation*}
where $\widetilde{C}=\frac{\sqrt{2}\sqrt{L_{\nabla h}+C^{-1}\|A\|^{2}}}{\sqrt{\left(1+1/N\right)^{1-\alpha}-1}}\sqrt{F_{N}(x_{N})-F^{*}+\mu_{N} \ell^{2}}$. 
\end{theorem}
\begin{proof}
Since $(x_{j})$ is bounded, there exists $m\geq 0$ such that $\|x_{j}\|\leq m$ for all $j\geq 1$. Define  $S= \|A\|m \mathbb{B}$. Then, by Proposition~\ref{prop_prox_w_convex}, there exist $\ell\geq 0$ and $\mu_{0}\in\left]0,1\right[$ such that for all $\mu\in\left]0,\mu_{0}\right[$ and $y\in S$,
\begin{align}
\label{grad_mor_env}
\|\nabla g_{\mu}(y)\|=\dfrac{1}{\mu}\|y-\prox_{\mu g}(y)\|\leq \ell.
\end{align}
Let $N\geq 1$ such that  $\mu_{k}<\mu_{0}$ and $\left(1+1/N\right)^{1-\alpha}-1<1/2$ for all $k\geq N$. By \cite[Lemma~4.1]{bohm2021variable}, the fact that $\{Ax_{k}\}\subset S$, and \eqref{grad_mor_env}, we have that for all $k\geq N$
\begin{align*}
F_{k+1}(x_{k+1})&\leq F_{k}(x_{k+1})+\dfrac{1}{2}(\mu_{k}-\mu_{k+1})\dfrac{\mu_{k}}{\mu_{k+1}}\|\nabla g_{\mu_{k}}(Ax_{k+1})\|^{2}\\
&\leq F_{k}(x_{k+1})+\dfrac{1}{2}(\mu_{k}-\mu_{k+1})\dfrac{\mu_{k}}{\mu_{k+1}}\ell^{2}\\
&\leq F_{k}(x_{k+1})+(\mu_{k}-\mu_{k+1})\ell^{2},
\end{align*}
where we have used that $\mu_{k}\leq 2^{\alpha}\mu_{k+1}\leq 2\mu_{k+1}$. Thus, from \eqref{des_Fk}, for all $k\geq N$,
\begin{align*}
F_{k+1}(x_{k+1})\leq F_{k}(x_{k})-\dfrac{\gamma_{k}}{2}\|P_{V}\nabla F_{k}(x_{k})\|^{2}+(\mu_{k}-\mu_{k+1})\ell^{2}.
\end{align*}
By summing both sides of the last inequality from $k=N$ to $k=K$, we obtain
\begin{equation}\label{sum_gamma_k}
\begin{aligned}
\sum_{k=N}^{K}\dfrac{\gamma_{k}}{2}\|P_{V}\nabla F_{k}(x_{k})\|^{2}&\leq F_{N}(x_{N})-F_{K+1}(x_{K+1})+(\mu_{N}-\mu_{K+1})\ell^{2}\\
&\leq F_{N}(x_{N})-F^{*}+\mu_{N}\ell^{2}.
\end{aligned}
\end{equation}
Now, note that
\begin{align}
\label{sum_k_alpha}
\sum_{k=N}^{K}k^{-\alpha}\geq \int_{N}^{K+1}x^{-\alpha}dx \geq (K+1)^{1-\alpha}-N^{1-\alpha}.
\end{align}
Setting $\theta:= \left(1+1/N\right)^{1-\alpha}-1$, from Lemma \ref{claim_N}, \eqref{des_gammak}, \eqref{sum_gamma_k}, and \eqref{sum_k_alpha}, we get 
\begin{align*}
\min_{N\leq j\leq K}\|P_{V}\nabla F_{j}(x_{j})\|^{2}\leq \frac{1}{\theta K^{1-\alpha}}(2(L_{\nabla h}+C^{-1}\|A\|^{2})(F_{N}(x_{N})-F^{*}+\mu_{N}\ell^{2})),
\end{align*}
which implies the first inequality. The second inequality follows from \eqref{grad_mor_env}. \qed
\end{proof}
\section{Variable smoothing with epochs}\label{sec_epoch}
We observe that in Theorem~\ref{teo_conv_weak_sev}, the first bound states that during the first $k=O(\epsilon^{-\frac{2}{1-\alpha}})$ iterations, there exists an iteration $j$ at which the optimality condition holds within a tolerance of $\epsilon$. Nevertheless, this bound could have been satisfied at an earlier
iteration, that is, for some $j\ll \epsilon^{-\frac{2}{1-\alpha}}$. Consequently, the value of the second bound on $\|Ax_{j}-\prox_{\mu_{j}g}(Ax_{j})\|$ may not be particularly small. This drawback is addressed by the following variant of Algorithm~\ref{alg_w_convex_sev},inspired by \cite{bohm2021variable}, where the steps are organized into a series of epochs, each of which is twice as long as the previous one. 
\begin{algorithm}[H]
\caption{Projected variable smoothing with epochs}\label{alg_epoch}
\begin{algorithmic} 
\Require $x_{1}\in V$, $\alpha\in\left]0,1\right[$, $C>0$ such that $2\rho C\leq 1$, and tolerance $\epsilon>0$.
\For{$l=0,1,\ldots$}
\State Set $S_{l} \gets \infty$ and $j_{l} \gets 2^{l}$
\For{$k=2^{l},2^{l}+1,\ldots,2^{l+1}-1$}
\State Set $\mu_{k}\gets Ck^{-\alpha}$, set $L_{k}$ as in \eqref{cons_L_k}, set $\gamma_{k}\gets 1/L_{k}$.
\State $x_{k+1}\gets P_{V}(x_{k}-\gamma_{k}\nabla F_{k}(x_{k}))$
\If{$\|\nabla F_{k+1}(x_{k+1})\|\leq S_{l}$}
\State $S_{l}\gets \|\nabla F_{k+1}(x_{k+1})\|$ and $j_{l}\gets k+1$
\If{$S_{l}\leq \epsilon$ and $\|Ax_{k+1}-\prox_{\mu_{k+1}g}(Ax_{k+1})\|\leq \epsilon^{\frac{2\alpha}{1-\alpha}}$}
\State STOP
\EndIf
\EndIf
\EndFor
\EndFor
\end{algorithmic}
\end{algorithm}
\noindent The next result recovers \cite[Theorem~4.2]{bohm2021variable} for $\alpha=1/3$, $C=(2\rho)^{-1}$, and $V=\H$. 
\begin{theorem}
In the context of problem~\eqref{prob_w_conv}, assume that $g$ is $L_{g}$-Lipschitz. Then, given a tolerance $\epsilon>0$, Algorithm~\ref{alg_epoch} generates an iterate $x_{j}$ for some $j=O(\epsilon^{-\frac{2}{1-\alpha}})$ such that $\|Ax_{j}-\prox_{\mu_{j}g}(Ax_{j})\|\leq\epsilon^{\frac{2\alpha}{1-\alpha}}$ and 
\begin{equation*} 
d(-\nabla h(x_{j}),A^{*}\partial g(\prox_{\mu_{j}g}(Ax_{j}))+N_{V}(x_{j}))\leq\epsilon.
\end{equation*}
\end{theorem}
\begin{proof} By \eqref{des_pv_grad}, we deduce that
\begin{align*}
\sum_{k=2^{l}}^{2^{l+1}-1}\frac{\gamma_{k}}{2}\|P_{V}\nabla F_{k}(x_{k})\|^{2}\leq F_{1}(x_{1})-F^{*}+C L_{g}^{2}.
\end{align*}
Now, note that
\begin{align*}
\sum_{k=2^{l}}^{2^{l+1}-1}k^{-\alpha}\geq \int_{2^{l}}^{2^{l+1}}x^{-\alpha}dx=\frac{1}{1-\alpha}(2^{1-\alpha}-1)(2^{l})^{1-\alpha}.
\end{align*}
Then, from \eqref{des_gammak}, we obtain that
\begin{align*}
\min_{2^{l}\leq j\leq 2^{l+1}-1}\|P_{V}\nabla F_{j}(x_{j})\|\leq \widetilde{C}(2^{l})^{\frac{\alpha-1}{2}},
\end{align*}
where $\widetilde{C}=\sqrt{\frac{2(1-\alpha)}{2^{1-\alpha}-1}}\sqrt{L_{\nabla h}+C^{-1}\|A\|^{2}}\sqrt{F_{1}(x_{1})-F^{*}+CL_{g}^{2}}$. Thus, by \eqref{des_key_1}, we conclude that
\begin{align}
\label{des_crit_1}
\min_{2^{l}\leq j\leq 2^{l+1}-1}d(-\nabla h(x_{j}),A^{*}\partial g(z_{j})+N_{V}(x_{j}))\leq \widetilde{C}(2^{l})^{\frac{\alpha-1}{2}},
\end{align}
where $z_{j}=\prox_{\mu_{j}g}(Ax_{j})$. Moreover, for $2^{l}\leq j\leq 2^{l+1}-1$, we have from \cite[Lemma~3.3]{bohm2021variable} that
\begin{align}
\label{des_fact_1}
\|Ax_{j}-z_{j}\|\leq \mu_{j}L_{g}=Cj^{-\alpha}L_{g}\leq C(2^{l})^{-\alpha}L_{g}.
\end{align}
From \eqref{des_crit_1} and \eqref{des_fact_1}, it follows that Algorithm~\ref{alg_epoch} must stop before the end of epoch $l$, that is, before $2^{l+1}$ iterations have been completed, where $l$ is the first nonnegative number such that $\widetilde{C}(2^{l})^{\frac{\alpha-1}{2}}\leq \epsilon$ and $C(2^{l})^{-\alpha}L_{g}\leq \epsilon^{\frac{2\alpha}{1-\alpha}}$, which are equivalent to $\widetilde{C}^{\frac{2}{1-\alpha}}\epsilon^{-\frac{2}{1-\alpha}}\leq 2^{l}$ and $C^{1/\alpha}L_{g}^{1/\alpha}\epsilon^{-\frac{2}{1-\alpha}}\leq 2^{l}$. Hence, the algorithm terminates after at most $2\max\{\widetilde{C}^{\frac{2}{1-\alpha}},C^{1/\alpha}L_{g}^{1/\alpha}\}\epsilon^{-\frac{2}{1-\alpha}}$ iterations. \qed
\end{proof}

\section{Supremum of weakly convex functions}
\label{sec_mor_sup}
Supremum functions play a prominent role in optimization. They models several phenomena, ranging from stochastic optimization to equilibrium problems. However, they are usually nonsmooth; hence, their optimization presents significant challenges from both a theoretical and numerical perspective.  In this section, we extend the analysis from \cite{MR4279933}. Hence, we study the Moreau envelope and the proximity operator of the supremum function: 
	\begin{align}
		\label{fun_sup}
		f(x) = \sup_{c\in \C} f_{c}(x) \quad \text{for all }x\in \H,
	\end{align}
	where $\H$ is a real Hilbert space, $\C\neq\emptyset$ is a convex subset of some topological vector space and $f_{c}\in\Gamma_{\rho_{c}}(\H)$ with $\rho_{c}\geq 0$ for all $c\in\C$.
	\begin{lemma}
		\label{sup_weakconv}
		Let $(f_{c})_{c\in\C}$ be a family of functions such that $f_{c}\in\Gamma_{\rho_{c}}(\H)$ for all $c\in\C$. Assume that $f:=\sup_{c\in\C} f_{c}$ is proper and that  $\bar{\rho}:=\sup_{c\in\C}\rho_{c}<+\infty$. Then, $f\in\Gamma_{\bar{\rho}}(\H)$. More precisely,
		\begin{align}
			\label{def_h_rc}
			h:=f+\frac{\bar{\rho}}{2}\|\cdot\|^{2}\in\Gamma_{0}(\H)\textrm{ and } r_{c}:=f_{c}+\frac{\bar{\rho}}{2}\|\cdot\|^{2}\in\Gamma_{0}(\H)\textrm{ for all }c\in\C.
		\end{align}
	\end{lemma}
	\begin{proof}
		It is clear that $h_{c}:=f_{c}+\dfrac{\rho_{c}}{2}\|\cdot\|^{2}\in\Gamma_{0}(\H)$ for all $c\in\C$.
Moreover, 
		\begin{align}
			\label{def_h}
			h:=f+\frac{\bar{\rho}}{2}\|\cdot\|^{2} = \sup_{c\in\C} f_{c}+\frac{\bar{\rho}}{2}\|\cdot\|^{2}=\sup_{c\in\C} (h_{c}+\frac{(\bar{\rho}-\rho_{c})}{2}\|\cdot\|^{2}).
		\end{align}
		Since $h_{c}$ belong to $\Gamma_0(\H)$, and $\bar{\rho}-\rho_{c}\geq 0$ for all $c\in\C$, we obtain that $r_{c}:=h_{c}+\frac{(\bar{\rho}-\rho_{c})}{2}\|\cdot\|^{2} = f_{c}+\frac{\bar{\rho}}{2}\|\cdot\|^{2}$ belongs to $\Gamma_0(\H)$  for all $c\in\C$. Since $\sup_{c\in\C}f_{c}$ is proper, we get that $h=\sup_{c\in\C} r_{c}\in\Gamma_{0}(\H)$. Hence, by \eqref{def_h}, $f\in\Gamma_{\bar{\rho}}(\H)$. \qed
	\end{proof}
	The next result provides a formula for the Moreau envelope of the supremum function in \eqref{fun_sup} as the supremum of the Moreau envelope of the data functions $f_{c}$. 	Whenever the data functions are convex, it recovers \cite[Theorem~3.1]{MR4279933}.  
	\begin{theorem}
		\label{teo_sup}
		Let $\C\neq\emptyset$ be a convex set and $(f_{c})_{c\in\C}$ a family of functions such that $f_{c}\in\Gamma_{\rho_{c}}(\H)$ for all $c\in\C$. Assume that $f=\sup_{c\in\C} f_{c}$ is proper,  $\bar{\rho}:=\sup_{c\in\C}\rho_{c}<+\infty$, and that $c\mapsto f_{c}(x)$ is concave for all $x\in \H$. Then, for all $\mu\in\left]0,1/\bar{\rho}\right[$, we have
		\begin{align}
			\label{result_teo_sup}
f_{\mu}(x)=\sup_{c\in\C}(f_{c})_{\mu}(x)\quad\text{for all } x\in \H.
		\end{align}
		Moreover, if $\C$ is compact and the function $c\mapsto f_{c}(x)$ is upper semicontinuous for all $x\in\H$, then the supremum in \eqref{result_teo_sup} is attained.
	\end{theorem}
	
	\begin{proof}
		Let $\mu\in ]0,1/\bar{\rho}[$ and $x\in\H$. By Lemma~\ref{sup_weakconv}, we have that $f\in\Gamma_{\bar{\rho}}(\H)$ and $h=\sup_{c\in\C}r_{c}$, where $h$ and the functions $r_{c}$ are defined by \eqref{def_h_rc}.
		Thus, from \eqref{def_h_rc}, \cite[Theorem~3.4(b)]{MR1230710}, and \cite[Theorem~3.1]{MR4279933}, it follows that
		\begin{align*}
			f_{\mu}(x)&=-\frac{1}{2}\bar{\rho}(1-\mu\bar{\rho})^{-1}\|x\|^{2} + h_{\frac{\mu}{1-\mu\bar{\rho}}}((1-\mu\bar{\rho})^{-1}x)\\
			&=\sup_{c\in\C}-\frac{1}{2}\bar{\rho}(1-\mu\bar{\rho})^{-1}\|x\|^{2}+(r_{c})_{\frac{\mu}{1-\mu\bar{\rho}}}((1-\mu\bar{\rho})^{-1}x)\nonumber\\
			&=\sup_{c\in\C}(f_{c})_{\mu}(x)\nonumber.
		\end{align*}
		Now, suppose that $\C$ is compact and the function $c\mapsto f_{c}(x)$ is upper semicontinuous for all $x\in\H$. Then for every $x\in\H$ the function $c\mapsto (f_{c})_{\mu}(x)$ is upper semicontinuous since it is the infimum of upper semicontinuous functions. Thus, the supremum in \eqref{result_teo_sup} is attained. \qed
	\end{proof}

	The next result is a generalization of \cite[Theorem~3.5]{MR4279933} and provide a formula for the gradient of the Moreau envelope and the proximity operator of  \eqref{fun_sup}. Given $\mu>0$, $f=\sup_{c\in\C}f_{c}$, and $x\in\H$, we define the following set
	\begin{align*}
		\C_{\mu}^{f}(x):=\left\{c\in\C:f_{\mu}(x)=(f_{c})_{\mu}(x)\right\}.
	\end{align*}
	\begin{theorem}
		\label{prop_dif_env_sup}
		Let $\C\neq\emptyset$ be a compact convex set and $(f_{c})_{c\in\C}$ be a family of functions such that $f_{c}\in\Gamma_{\rho_{c}}(\H)$ for all $c\in\C$. Suppose that $f=\sup_{c\in\C} f_{c}$ is proper,   $\bar{\rho}:=\sup_{c\in \C} \rho_{c}<+\infty$, and that $c\mapsto f_{c}(x)$ is concave and u.s.c. for all $x\in \H$. Then, for all $\mu\in ]0,1/\bar{\rho}[$ and $x\in\H$, the set $\C_{\mu}^{f}(x)$ is nonempty,
$$
\prox_{\mu f}(x)=\prox_{\mu f_{c}}(x), \textrm{ and } \nabla f_{\mu}(x)=\frac{x-\prox_{\mu f_{c}}(x)}{\mu}\quad\text{for all } c\in\C_{\mu}^{f}(x).
$$
	\end{theorem}
	
	\begin{proof}
		Let $\mu\in ]0,1/\bar{\rho}[$ and $x\in\H$. First, by Theorem~\ref{teo_sup}, the set $\C_{\mu}^{f}(x)$ is nonempty. Thus, by Lemma~\ref{sup_weakconv} and \cite[Theorem~3.4(d)]{MR1230710}, $f_{\mu}\in C^{1,1}$ and
$$
			\nabla f_{\mu}(x)=\frac{1}{\mu}(x-\prox_{(\frac{\mu}{1-\mu\bar{\rho}}) h}((1-\mu \bar{\rho})^{-1}x))\text{ for all }x\in\H,
$$
		where $h$ is defined by \eqref{def_h_rc}.
		Now, since the function $c\mapsto r_{c}(x)$ is concave and u.s.c. for all $x\in\H$, $\C$ is compact, and $h=\sup_{c\in\C}r_{c}$, then from \cite[Theorem~3.5]{MR4279933}, it follows that for all $c\in\C_{\frac{\mu}{1-\mu\bar{\rho}}}^{h}((1-\mu\bar{\rho})^{-1}x)$
		\begin{align*}
			\prox_{(\frac{\mu}{1-\mu\bar{\rho}}) h}((1-\mu\bar{\rho})^{-1}x)&=\prox_{(\frac{\mu}{1-\mu\bar{\rho}}) r_{c}}((1-\mu\bar{\rho})^{-1}x),
		\end{align*}
		which yields from \cite[Theorem~3.4(d)]{MR1230710} and \eqref{def_h_rc} that
		\begin{align*}
			\nabla f_{\mu}(x)=\frac{1}{\mu}(x-\prox_{\mu f_{c}}(x))\quad\text{for all }c\in\C_{\frac{\mu}{1-\mu\bar{\rho}}}^{h}((1-\mu\bar{\rho})^{-1}x).
		\end{align*}
		Since $r_{c}=f_{c}+\frac{\bar{\rho}}{2}\|\cdot\|^{2}$ and $h=f+\frac{\bar{\rho}}{2}\|\cdot\|^{2}$, we see from \cite[Theorem~3.4(b)]{MR1230710} that $\C_{\frac{\mu}{1-\mu\bar{\rho}}}^{h}\left(\frac{x}{1-\mu\bar{\rho}}\right) = \C_{\mu}^{f}(x)$. Finally, $\nabla f_{\mu}(x)=\frac{1}{\lambda}(x-\prox_{\mu f}(x))$, whence $\prox_{\mu f}(x)=\prox_{\mu f_{c}}(x)$ for all $x\in\H$ and $c\in\C_{\mu}^{f}(x)$. \qed
	\end{proof}
 In the following subsections, based on Theorem~\ref{prop_dif_env_sup}, we compute the proximity of two particular supremum functions which are weakly convex.
 \subsection{Supremum functions of weakly convex with quadratic functions}
 Consider the following function
\begin{align}
\label{sup_w_1}
\mathbf{x}\in\H^{N}\mapsto f(\mathbf{x})=\sup_{p\in\Delta_{N}}-\sum_{i=1}^{N}p_{i}\|x_{i}-\xi_{i}\|^{2}=\max_{1\leq i\leq N}-\|x_{i}-\xi_{i}\|^{2},
\end{align}
where $\xi_{i}\in\H$ for all $i\in\{1,\ldots,N\}$. The function $f$ in \eqref{sup_w_1} appears in discrete versions of DRO problems where the ambiguity set 
 is the the probability simplex $\Delta_{N}$, and the cost function is a weakly convex quadratic function. 
\begin{proposition}
\label{p_prox_sup_w_quad}
Let $\mathbf{x}\in\H^{N}$ and let $\mu\in\left]0,1/2\right[$. Then, $f$ is $\rho$-weakly convex with $\rho=2$ and
$$
\prox_{\mu f}(\mathbf{x})=\left(\dfrac{x_{i}-2\mu \bar{p}_{i}\xi_{i}}{1-2\mu \bar{p}_{i}}\right)_{i=1}^{N},
$$
 where $\bar{p}\in\Delta_{N}$ is a solution to
 \begin{align}
\label{prob_aux_w}
&\max_{p\in\Delta_{N}} \phi(p):=\sum_{i=1}^{N}\left(\frac{p_{i}}{2\mu p_{i}-1}\right)\alpha_{i} \textrm{ and } \alpha_{i}=\|x_{i}-\xi_{i}\|^{2} \textrm{ for all } i\in\{1,\ldots,N\}.
\end{align}
\end{proposition}
\begin{proof}
Define $f_{i}(x):=-\|x-\xi_{i}\|^{2}$ for all $i\in\{1,\ldots,N\}$ and $\mathbf{x}\in\H^{N}\mapsto g_{p}(\mathbf{x})=\sum_{i=1}^{N}p_{i}f_{i}(x_{i})$ for all $p\in\Delta_{N}$. We claim that $g_{p}$ is weakly convex with parameter $\rho=2$ for all $p\in\Delta_{N}$. Note that
 \begin{align*}
\sum_{i=1}^{N}p_{i}f_{i}(x_{i})+\|\mathbf{x}\|^{2}
=\sum_{i=1}^{N}p_{i}(f_{i}(x_{i})+\|x_{i}\|^{2})+\sum_{i=1}^{N}(1-p_{i})\|x_{i}\|^{2}.
 \end{align*}
 Then, since $f_{i}+\|\cdot\|^{2}$ is convex and $p_{i}\leq 1$ for all $i\in\{1,\ldots,N\}$,  $g_{p}+\|\cdot\|^{2}$ is convex for all $p\in\Delta_{N}$. Thus, by Lemma~\ref{sup_weakconv}, $f=\sup_{p\in\Delta_{N}}g_{p}$ is proper and $\rho$-weakly convex with $\rho=2$. On the other hand, $p\mapsto g_{p}(\mathbf{x})$ is concave and u.s.c. for all $\mathbf{x}\in\H^{N}$. Thus, from Theorem~\ref{prop_dif_env_sup}, for every $\mu\in\left]0,1/2\right[$,
 \begin{align}
\label{prox_gp_w_1}
\prox_{\mu f}(\mathbf{x})=\prox_{\mu g_{\bar{p}}}(\mathbf{x}),\quad\text{with }\bar{p}\in\operatorname{argmax}_{p\in\Delta_{N}}(g_{p})_{\mu}(\mathbf{x}).
 \end{align}
 Now, observe that for all $i\in\{1,\ldots,N\}$ and $x\in\H$, $y=\prox_{\mu (p_{i}f_{i})}(x)$ if and only if $y=\frac{x-2\mu p_{i}\xi_{i}}{1-2\mu p_{i}}$. Then,  we have that
 \begin{equation*}
\prox_{\mu g_{\bar{p}}}(\mathbf{x})=\left(\dfrac{x_{i}-2\mu \bar{p}_{i}\xi_{i}}{1-2\mu \bar{p}_{i}}\right)_{i=1}^{N}.
 \end{equation*}
 Let us calculate $(g_{p})_{\mu}(\mathbf{x})$. By the above formula, we obtain that
\begin{equation*}
(g_{p})_{\mu}(\mathbf{x})=g_{p}(\prox_{\mu g_{p}}(\mathbf{x}))+\frac{1}{2\mu}\|\mathbf{x}-\prox_{\mu g_{p}}(\mathbf{x})\|^{2}=\sum_{i=1}^{N}\left(\frac{p_{i}}{2\mu p_{i}-1}\right)\|x_{i}-\xi_{i}\|^{2}.
\end{equation*}
Set $\alpha_{i}=\|x_{i}-\xi_{i}\|^{2}$ for all $i\in\{1,\ldots,N\}$. Hence, in order to find $\bar{p}$ in \eqref{prox_gp_w_1}, we need to solve problem \eqref{prob_aux_w}. Since $\phi$ is continuous and $\Delta_{N}$ is a nonempty compact set, problem \eqref{prob_aux_w} has solutions. \qed
\end{proof}

In the context of problem~\eqref{prob_aux_w}, if $\alpha_{i}=0$ for some $i\in\{1,\ldots,N\}$, then a solution of problem~\eqref{prob_aux_w} is $p=e_{i}$, where $e_{i}\in\R^{N}$ is the canonical vector in $\R^{N}$. Then, we can assume that $\alpha_{i}>0$ for all $i\in\{1,\ldots,N\}$. We note that
\begin{equation*}
(\nabla^{2}\phi(p))_{ij}=\begin{cases}
4\mu\alpha_{i}(2\mu p_{i}-1)^{-3} & \text{ if } i=j\\
0 & \text{ if }i\neq j
\end{cases}\quad\text{for all } i,j\in\{1,\ldots,N\}.
\end{equation*}
Then, $\nabla^{2}(-\phi)(p)$ is a positive definite matrix for all $p\in\Delta_{N}$ since $\mu<1/2$. Thus $-\phi$ is strictly convex on $\Delta_{N}$. In addition, $\Delta_{N}$ is a compact convex set and $-\phi$ is continuous. Therefore, the problem~\eqref{prob_aux_w} has an unique solution.

Now, we provide a solution of \eqref{prob_aux_w}. The proof is given in Section \ref{Annex-C}.
 \begin{proposition}
		\label{prop_prob_aux_w}
		In the context of \eqref{prob_aux_w}, let $\{\ell_{i}\}_{i=1}^{N}$ such that $\alpha_{\ell_{1}}\geq\cdots\geq\alpha_{\ell_{N}}$ with $\alpha_{i}>0$ for all $i\in\{1,\ldots,N\}$ and let $I_{i}:=\{\ell_{1},\ldots,\ell_{i}\}$ for all $i\in\{1,\ldots,N\}$. Define
		\begin{align}
			\label{def_k_w}
			k:=\min\{i\in\{0,\ldots,N-1\}:(N-i-2\mu)\sqrt{\alpha_{\ell_{i+1}}}<\sum_{j\notin I_{i}}\sqrt{\alpha_{j}}\},
		\end{align}
		where $I_{0}:=\emptyset$. Then, $\bar{p}\in \R^{N}$ defined by
		\begin{align}
			\label{sol_prob_aux_w}
			\bar{p}_{i}=
			\begin{cases}
				0 & \text{if}\,\, i\in I_{k}\\
				\dfrac{1}{2\mu}[1-\frac{(N-k-2\mu)\sqrt{\alpha_{i}}}{\sum_{j\notin I_{k}}\sqrt{\alpha_{j}}}]& \text{if}\,\, i\notin I_{k}
			\end{cases}\quad\text{ for all } i\in\{1,\ldots,N\}
		\end{align}
		is the solution to problem \eqref{prob_aux_w}.
	\end{proposition}
 
\subsection{Supremum function weakly convex with affine functions}
Inspired in \cite[Theorem~3]{pmlr-v89-sun19b}, we consider now the function
 \begin{align}
\label{fun_sup_weak}
x\in\R^{n}\mapsto f(x)=\sup_{c\in\C}\sum_{i=1}^{N}c_{i}(\langle a_{i},x\rangle+b_{i})-\sigma\|x\|^{2},
 \end{align}
 where $a_{i}\in\R^{n}$, $b_{i}\in\R$, $\sigma>0$, and $\C\subset\Delta_{N}$ is a nonempty closed convex set. Clearly, $f$ is $\rho$-weakly convex with $\rho=2\sigma$. Note that the function $f$ appears in the discrete version of the DRO problem when the ambiguity set is the set $\C$ and the cost function is weakly convex involving affine functions. The next result provides an algorithm for the proximity of the function $f$ in \eqref{fun_sup_weak}.
  \begin{proposition}
\label{p_prox_sup_w_aff}
Let $x\in\R^{n}$ and $\mu\in]0,\frac{1}{2\sigma}[$. Let $A=\begin{bmatrix} a_{1}|\cdots|a_{N} \end{bmatrix}^{\top}$, $b=\begin{pmatrix} b_{1}\cdots b_{N}\end{pmatrix}^{\top}$, and  $\gamma>0$ such that $\frac{\gamma\mu\|AA^{\top}\|}{1-2\sigma\mu}<1$. Let $(\alpha_{k})\subset [0,1]$ such that $\sum_{k\in\mathbb{N}}\alpha_{k}(1-\alpha_{k})=+\infty$ and let $c^{0}\in\R^{N}$. For every $k\in\mathbb{N}$, we consider  
\begin{equation*}
y^{k}=\frac{1}{1-2\sigma \mu}(x-\mu A^{\top}c^{k}),  z^{k}=c^{k}+\gamma(Ay^{k}+b), 
c^{k+1}=\alpha_{k}P_{\C}(z^{k})+(1-\alpha_{k})c^{k}.
\end{equation*}
Then, $y^{k}\rightarrow \prox_{\mu f}(x)$.
 \end{proposition}
 \begin{proof} For every $c\in\C$, consider $f_{c}(x)=\sum_{i=1}^{N}c_{i}(\langle a_{i},x\rangle+b_{i})-\sigma\|x\|^{2}$. It is clear that $x\mapsto f_{c}(x)$ is weakly convex with parameter $2\sigma$ and that $c\mapsto f_{c}(x)$ is concave and upper semicontinuous for all $x\in\R^{n}$.  Moreover, it is clear that  $f=\sup_{c\in\C}f_{c}$ is proper.  Thus, by Theorem~\ref{prop_dif_env_sup}, we have that
 \begin{align}
 \label{prox_sup_w_aff}
\prox_{\mu f}(x)=\prox_{\mu f_{\bar{c}}}(x)\quad\text{with }\bar{c}\in\operatorname{argmax}_{c\in\C}(f_{c})_{\mu}(x).
 \end{align}
Moreover, for every $c\in\C$,  $y=\prox_{\mu f_{c}}(x)$ if and only if  $y=\frac{1}{1-2\sigma \mu}(x-\mu A^{\top}c)$. In order to find $\bar{c}\in\C$ in \eqref{prox_sup_w_aff}, we need to solve the following minimax problem
 \begin{equation*}
\max_{c\in\C}(f_{c})_{\mu}(x)=\max_{c\in\C}\inf_{y\in\R^{n}} f_{c}(y)+\dfrac{1}{2\mu}\|y-x\|^{2}.
 \end{equation*}
 The optimality condition of the above problem is $Ay+b\in N_{\C}(c)$ and $y=\prox_{\mu f_{c}}(x)$. Hence, we need to find $c\in\C$ such that $c=P_{\C}(c+\gamma(Ay_{c}+b))$, where $y_{c}=(1-2\sigma \mu)^{-1}(x-\mu A^{\top}c)$. Then, by defining $c\in\C\mapsto\mathcal{F}(c)=c+\gamma(Ay_{c}+b)$, we conclude that
\begin{align}
\label{carac_c}
c\in\operatorname{argmax}_{c\in\C}(f_{c})_{\mu}(x)\Leftrightarrow c=P_{\C}(\mathcal{F}(c)).
\end{align}
We claim that $\mathcal{F}$ is nonexpansive. Indeed, let $c_{1}$ and $c_{2}$ in $\C$. Then
\begin{align*}
\mathcal{F}(c_{1})-\mathcal{F}(c_{2})&=c_{1}-c_{2}-\dfrac{\gamma\mu AA^{\top}}{1-2\sigma\mu}(c_{1}-c_{2})=\left(I-\dfrac{\gamma\mu AA^{\top}}{1-2\sigma\mu}\right)(c_{1}-c_{2}).
\end{align*}
Let $B:=\dfrac{\gamma\mu AA^{\top}}{1-2\sigma\mu}$. It is easy to see that $\|I-B\|_{2}\leq 1$.  Define $T=P_{\C}\circ \mathcal{F}$. Then, since $P_{\C}$ is nonexpansive, we obtain that $T$ is nonexpansive. Hence,  we deduce that $c^{k+1}=\alpha_{k}Tc^{k}+(1-\alpha_{k})c^{k}$. From \cite[Theorem~5.15]{MR3616647}, there exists $\bar{c}\in\text{Fix}\,T$ such that $c^{k}\rightarrow \bar{c}$ and  $y^{k}\rightarrow \prox_{\mu f_{\bar{c}}}(x)$. Moreover, by \eqref{prox_sup_w_aff} and \eqref{carac_c}, $\prox_{\mu f_{\bar{c}}}(x)=\prox_{\mu f}(x)$. Therefore, $y^{k}\rightarrow \prox_{\mu f}(x)$. \qed
 \end{proof}
  \section{Applications}\label{sec_aplic}
  \subsection{Distributionally Robust Optimization}
Let $(\Omega,\mathcal{A})$ be a measurable space, let $\xi\colon \Omega\rightarrow \Xi\subset\R^{m}$ be a random vector with $\Xi\subset\R^{m}$ being an uncertainty set, let $\mathcal{P}$ be a nonempty closed convex subset of probability measures on $(\Omega,\mathcal{A})$ supported on $\Xi$ (called \textit{ambiguity set}), let $Q\subset\R^{n}$ be a compact convex set, and let $F\colon \R^{n}\times \Xi\rightarrow \R$ be a measurable function. The problem 
\begin{align}
\label{dro_cont}
\displaystyle\min_{x\in Q}\displaystyle\sup_{\mathbb{P}\in \mathcal{P}} \mathbb{E}_{\mathbb{P}}[F(x,\xi)]:=\int_{\Omega}F(x,\xi(\omega))d\mathbb{P}(\omega).
\end{align}
is called \textit{distributionally robust optimization} (DRO) problem \cite[eq~4.1.3]{sun2021robust}. The ambiguity set represents the information about the probability distributions. The problem~\eqref{dro_cont} covers the stochastic optimization problem when the ambiguity set is a singleton, and the robust optimization problem when the ambiguity set includes all possible distributions.  When $\Xi$ is finite, that is, $\Xi=\{\xi_{1},\ldots,\xi_{N}\}$ and denoting $p_{i}=\mathbb{P}(\{\omega\in\Omega:\xi(\omega)=\xi_{i}\})$ and $f_{i}=F(\cdot,\xi_{i})$, for every $\mathbb{P}\in\mathcal{P}$ and $i\in\{1,\ldots,N\}$, we have that the problem~\eqref{dro_cont} reduces to
\begin{align}
\label{dro_disc}
\min_{x\in Q}\sup_{p\in\mathcal{P}}\sum_{i=1}^{N}p_{i}f_{i}(x),
\end{align}
where $\mathcal{P}\subset\Delta_{N}:=\{p\in\R_{+}^{N}:\sum_{i=1}^{N}p_{i}=1\}$. When $Q=V\cap \mathcal{B}$, where $V$ is a closed vector subspace of $\R^{n}$ and $\mathcal{B}$ is a compact convex set, \eqref{dro_disc} reduces to
\begin{align}
\label{dro_disc_2}
\min_{x\in V\cap \mathcal{B}}\widetilde{g}(x),
\end{align}
where $\widetilde{g}(x)=\sup_{p\in\mathcal{P}}\sum_{i=1}^{N}p_{i}f_{i}(x)$. Note that if $f_{i}$ is l.s.c. for all $i\in\left\{1,\ldots,N\right\}$, then $\widetilde{g}$ is l.s.c.. Then, we can use the penalty method described in Section~\ref{penalty_method} for solving \eqref{dro_disc_2}. That is, we need to solve the penalized problem:
\begin{align*}
\min_{x\in V}\lambda P(x)+\widetilde{g}(x),
\end{align*}
where $\lambda>0$ and $P\colon\R^{n}\rightarrow\R$ is a l.s.c.  function such that $P(x)\geq 0$ for all $x\in\R^{n}$ and $P(x)=0$ if and only if $x\in \mathcal{B}$. When $P(x)=\frac{1}{2}d^2(x,\mathcal{B})$, the penalized problem is equivalent to
\begin{align}
\label{dro_disc_4}
\min_{x\in V}h(x)+\widetilde{g}(x),
\end{align}
where $h(x)=\frac{\lambda}{2}d^2(x,\mathcal{B})$. Then, $\nabla h(x)=\lambda(x-P_{\mathcal{B}}(x))$ and $\nabla h$ is Lipschitz with constant $\lambda$. On the other hand, when $f_{i}(x)=\langle a_{i},x\rangle+b_{i}-\sigma\|x\|^{2}$, where $a_{i}\in\R^{n}$, $b_{i}\in\R$, and $\sigma>0$, the proximity operator of $\widetilde{g}$ can be calculated using the algorithm in Proposition~\ref{p_prox_sup_w_aff}. Thus, \eqref{dro_disc_4} is a particular case of \eqref{prob_w_conv}. When $f_{i}(x)=\langle a_{i},x\rangle+b_{i}-\sigma\|x\|^{2}$, it is straightforward to prove that the objective function in \eqref{dro_disc_4} is coercive if $\lambda>2\sigma$ and hence, in this case, \eqref{dro_disc_4} has solutions for all $\lambda>2\sigma$.
We consider now the following DRO model \cite[Example~4]{MR4361979}
\begin{equation}
\label{prob_risk_averse}
\min_{\mathbf{x}=(x_{1},\ldots,x_{N})\in\R^{nN}}\max_{p\in\mathcal{P}} \sum_{i=1}^{N}p_{i}f_{i}(x_{i})\quad\text{s.t.}\quad \mathbf{x}\in\mathcal{N}\cap\bigtimes_{i=1}^{N}Q_{i},
\end{equation} 
where $\mathcal{N}\subset \R^{nN}$ is the nonanticipative set (vector subspace), $f_{i}\colon\R^{n}\rightarrow\R$ is the cost function in the scenario $i$, $Q_{i}$ is a compact convex subset, and $\mathcal{P}\subset\Delta_{N}$ is a nonempty closed convex ambiguity set. Note that \eqref{prob_risk_averse} is equivalent to
\begin{align}
\label{prob_risk_averse_2}
\min_{\mathbf{x}\in\mathcal{N}\cap \overline{Q}}g(\mathbf{x}),
\end{align}
where $\overline{Q}=\bigtimes_{i=1}^{N}Q_{i}$ and $g(\mathbf{x})=\sup_{p\in\mathcal{P}} \sum_{i=1}^{N}p_{i}f_{i}(x_{i})$. Observe that if $f_{i}$ is lower semicontinuous for all $i\in\{1,\ldots,N\}$, then $g$ also is lower semicontinuous. Similarly to \eqref{dro_disc_2}, in order to solve \eqref{prob_risk_averse_2}, we can use the penalty method described in Section~\ref{penalty_method} by solving the following problem
\begin{align}
\label{prob_risk_averse_3}
\min_{\mathbf{x}\in\mathcal{N}}\frac{\lambda}{2}d(\mathbf{x},\overline{Q})^{2}+g(\mathbf{x}),
\end{align}
where $\lambda>0$. In the case when $f_{i}(x)=-\|x-\xi_{i}\|^{2}$ and $\mathcal{P}=\Delta_{N}$, where $\xi_{i}\in\R^{n}$, the proximity operator of $g$ has a closed-form and is given in Proposition~\ref{p_prox_sup_w_quad} and Proposition~\ref{prop_prob_aux_w}. Moreover, in this case, we have that the objective function in \eqref{prob_risk_averse_3} is coercive if $\lambda>2$. Hence, in this case 
 \eqref{prob_risk_averse_3} has solutions for all $\lambda>2$. Thus, \eqref{prob_risk_averse_3} is a particular case of problem~\eqref{prob_w_conv}.

  \subsection{Constrained LASSO}
  A constrained LASSO problem with linear constraints has the following form:
  \begin{equation}
  \label{lasso_conv}
\min_{x\in \R^{n},  Rx=0}\|Bx-b\|^{2}+\|x\|_{1},
  \end{equation}
  where $\|\cdot\|_{1}$ is the $\ell^{1}$ norm which is a regularizer that induce sparsity in the solution vector $x$. The previous problem is convex and is a particular case of \eqref{prob_w_conv} when $V=\ker R$, $A$ is the identity, and $h(x)=\|Bx-b\|^{2}$. Another class of problems where the regularization with the norm $\|\cdot\|_{1}$ appears is in logistic regression \cite{shi2006lasso}. To use the norm $\|\cdot\|_{1}$ as a regularizer in \eqref{lasso_conv} tends to reduce the number of nonzero elements of the solution generating a bias. Therefore, nonconvex regularizers are often used to reduces bias. These covers the $\ell^{p}$-norms (with $p\in\left]0,1\right[$) which are not weakly convex, as well as several weakly convex regularizers, which now we illustrate. In \cite{zhang2010nearly} it is defined the minimax concave penalty (MCP) function: Given $\lambda>0$ and $\theta>0$, the MCP is the function $r_{\lambda,\theta}\colon\R\rightarrow\R_{+}$ defined by
  \begin{equation*}
r_{\lambda,\theta}(x)=\begin{cases}
\lambda |x|-\frac{x^{2}}{2\theta} & \text{if }\,\,|x|\leq\theta\lambda,\\
\frac{\theta\lambda^{2}}{2} & \text{otherwise }
\end{cases}
  \end{equation*}
  Observe that this function is $\theta^{-1}$-weakly convex. The proximity operator of $r_{\lambda,\theta}$ (called firm threshold) has a closed form and is given, for every $\gamma<\theta$, by
  \begin{equation*}
\prox_{\gamma r_{\lambda,\theta}}(x)=\begin{cases}
0 & \text{if }\,\, |x|<\gamma\lambda,\\
\frac{x-\lambda\gamma\sgn(x)}{1-\gamma/\theta} & \text{if }\,\,\gamma\lambda\leq |x|\leq \theta\lambda,\\
x & \text{if }\,\, |x|>\theta\lambda.
\end{cases}
  \end{equation*}
  Another weakly convex regularizer is the smoothly clipped absolute deviation (SCAD) \cite{MR1946581}, which is defined, for parameters $\lambda>0$ and $\theta>2$, by
  \begin{equation*}
r_{\lambda,\theta}(x)=\begin{cases}
\lambda|x| & \text{if }\,\,|x|\leq\lambda,\\
\frac{-x^{2}+2\lambda\theta|x|-\lambda^{2}}{2(\theta-1)} & \text{if }\,\, \lambda<|x|\leq\theta\lambda,\\
\frac{(\theta+1)\lambda^{2}}{2} & \text{if }\,\,|x|>\theta\lambda.
\end{cases}
  \end{equation*}
  This function is $(\theta-1)^{-1}$-weakly convex. The problem we have seen so far has the structure of \eqref{prob_w_conv} when $A$ is the identity. Let us consider the case $A\neq \Id$. We consider the \textit{Tukey biweight} function \cite{beaton1974fitting}, which is given by
  \begin{align*}
x\in\R^{n}\mapsto \sum_{i=1}^{m}\varphi(A_{i\bullet}x-b_{i}),
  \end{align*}
  where $\varphi(\theta)=\frac{\theta^{2}}{1+\theta^{2}}$, $A\in\R^{m\times n}$, and $A_{i\bullet}$ denotes the $i$th row of $A$. Note that the above function can be written as $g(Ax)$ with  $g(y)=\sum_{i=1}^{m}\varphi(y_{i}-b_{i})$ for $y\in\R^{m}$. Moreover, the Tukey biweight function is $\rho$-weakly convex with $\rho=6$.
  \subsection{Max Dispersion Problems}
 Consider the following max dispersion problem (or max-min location problem):
 \begin{align}
 \label{max_dis}
\max_{x\in S}\min_{1\leq i\leq N}\|x-u_{i}\|^{2},
 \end{align}
 where $S\subset\R^{n}$ is a compact subset and $u_{i}\in\R^{n}$ for all $i\in\{1,\ldots,N\}$. The problem~\eqref{max_dis} consists in to find a point in $S$ that is furthest from the given points $u_{1},\ldots,u_{N}$ and has the geometric interpretation of finding the largest Euclidian ball with center in $S$ and that does not enclose any given point. Note that $x\mapsto \min_{1\leq i\leq N}\|x-u_{i}\|^{2}$ is continuous since is the minimum of finitely many continuous functions. Then, problem~\eqref{max_dis} has solutions. The problem \eqref{max_dis} have been studied in \cite{jeyakumar2018exact} when the constraint $S$ is a polyhedral or a ball, and has applications in facility location, pattern recognition, among others \cite{dasarathy1980maxmin,johnson1990minimax}. In addition, it is shown in \cite{haines2013convex} that the max dispersion problem over a box constraint is a NP-hard problem.
 
Consider the case when $S=V\cap \mathcal{B}$, where $V\subset\R^{n}$ is a vector subspace and $\mathcal{B}\subset\R^{n}$ is a compact convex set. In this case, \eqref{max_dis} is equivalent to
 \begin{align}
 \label{max_dis_2}
\min_{x\in V\cap \mathcal{B}}\max_{1\leq i\leq N}-\|x-u_{i}\|^{2}.
 \end{align}
 In order to solve \eqref{max_dis_2}, we use the penalty method described in Section~\ref{penalty_method}. That is, we solve the following penalized problem
 \begin{align}
 \label{pen_max_dis}
\min_{x\in V}\lambda P(x)+\max_{1\leq i\leq N}-\|x-u_{i}\|^{2},
 \end{align}
     where $P$ is the penalty function which is lower semicontinuous and satisfies that $P(x)\geq 0$ for all $x\in\R^{n}$ and $P(x)=0$ if and only if $x\in \mathcal{B}$, while $\lambda>0$ is the penalty coefficient. Consider the penalty function $P(x)=\frac{1}{2}d(x,\mathcal{B})^{2}$. Then, the problem~\eqref{pen_max_dis} is equivalent to
 \begin{align}
 \label{app_w_1}
\min_{x\in V}h(x)+\widetilde{g}(x),
 \end{align}
 where $h(x)=\frac{\lambda}{2}d^2(x,\mathcal{B})$ and $\widetilde{g}$ is given by
 \begin{align}
 \label{fun_g_til_1}
x\in\R^{n}\mapsto \widetilde{g}(x)=\sup_{p\in\Delta_{N}}\sum_{i=1}^{N}p_{i}(\langle x,2u_{i}\rangle-\|u_{i}\|^{2})-\|x\|^{2}.
 \end{align}
 The problem~\eqref{app_w_1} is particular instance of \eqref{prob_w_conv} when $A=\Id$. Moreover, it is straightforward to prove that the objective function in \eqref{app_w_1} is coercive for all $\lambda>2$, hence \eqref{app_w_1} has solutions for all $\lambda>2$. Now, the function \eqref{fun_g_til_1} is a particular case of \eqref{fun_sup_weak} when $a_{i}=2u_{i}$, $b_{i}=-\|u_{i}\|^{2}$, $\sigma=1$, and $\C=\Delta_{N}$. Thus, we can use the algorithm in Proposition~\ref{p_prox_sup_w_aff} for computing the proximity operator of $\widetilde{g}$. Observe that the projection onto $\C=\Delta_{N}$ has an explicit form given in \cite{duchi2008efficient}. In addition,  $\nabla h(x)=\lambda(x-P_{\mathcal{B}}(x))$ is Lipschitz with constant $\lambda$. 
 
 On the other hand, since $V\subset\R^{n}$ is a vector subspace, then by \cite[Theorem~1.4]{rockafellar1997convex}, there exists $R\in\R^{m\times n}$ such that $V=\ker R$. Thus, since $\ran R$ is closed (since it is a vector subspace of finite dimension), then the projection onto $V$ is given by  $ P_{\ker R}=\Id-R^{\dagger}R$, where $R^{\dagger}$ is the generalized (or Moore-Penrose) inverse of $R$.

 Another equivalent formulation for the problem~\eqref{app_w_1} is the following
 \begin{align}
 \label{app_w_2}
\min_{\mathbf{x}=(x_{1},\ldots,x_{N})\in\R^{nN}} H(\mathbf{x})+g(\mathbf{x})\quad\text{s.t.}\quad\mathbf{x}\in V^{N}\cap\mathcal{D},
 \end{align}
 where $H\colon\R^{nN}\rightarrow\R$ is defined by $H(\mathbf{x})=h(x_{1})=\frac{\lambda}{2}d(x_{1},\mathcal{B})^{2}$, $\mathcal{D}=\{\mathbf{x}\in\R^{nN}:x_{1}=\cdots=x_{N}\}$ is the diagonal set, and $g$ is given by
 \begin{align}
\label{fun_g_1}
\mathbf{x}\in\R^{nN}\mapsto g(\mathbf{x})=\max_{1\leq i\leq N}-\|x_{i}-u_{i}\|^{2}.
 \end{align}
Note that $\mathbf{x}=(x,\ldots,x)$ is solution to \eqref{app_w_2} if and only if $x$ is solution to \eqref{app_w_1}. The problem~\eqref{app_w_2} is a particular case of \eqref{prob_w_conv} when $A=\Id$ and the closed vector subspace is $V^{N}\cap\mathcal{D}$. As mentioned before, the objective function in \eqref{app_w_2} is coercive in $\mathcal{D}$ for all $\lambda>2$, hence \eqref{app_w_2} has solutions for all $\lambda>2$. The proximity operator of \eqref{fun_g_1} has a closed form and is given by Proposition~\ref{p_prox_sup_w_quad} and Proposition~\ref{prop_prob_aux_w}.

The projection onto $V^{N}\cap\mathcal{D}$ can be calculated using the Dykstra's projection algorithm, which requires calculate the projection onto $V^{N}$ and onto $\mathcal{D}$. Recall that $V=\ker R$. By \cite[Proposition~29.3]{MR3616647}, we have that $P_{V^{N}}(\mathbf{x})=(P_{\ker R}(x_{i}))_{i=1}^{N}$ for all  $\mathbf{x}\in\R^{nN}$. On the other hand, the projection onto $\mathcal{D}$ is given by $P_{\mathcal{D}}(\mathbf{x})=(\frac{1}{N}\sum_{j=1}^{N}x_{j})_{i=1}^{N}$ for all $\mathbf{x}\in\R^{nN}$.

 \section{Numerical experiments}
 \label{sec_num}
 In this section, we use the Algorithm~\ref{alg_w_convex_sev} for solving the equivalent problems~\eqref{app_w_1} and \eqref{app_w_2} associated to the max dispersion problem. We consider $\mathcal{B}=B(0,r)$ with $r>0$ and $V=\ker R$ for some $R\in\R^{m\times n}$. Observe that the functions $\widetilde{g}$ and $g$ defined in \eqref{fun_g_til_1} and \eqref{fun_g_1} respectively, are $\rho$-weakly convex with $\rho=2$. 
 In the context of problem~\eqref{app_w_1}, we have that $L_{\nabla h}=\lambda$ and $A=\Id$. Then, $L_{k}$ defined in \eqref{cons_L_k} is $L_{k}=\lambda+\dfrac{1}{\mu_{k}}=\dfrac{1+\lambda\mu_{k}}{\mu_{k}}$ and $\gamma_{k}=1/L_{k}=\dfrac{\mu_{k}}{1+\lambda\mu_{k}}$. Hence, for all $k\geq 1$ (where $F_{k}$ is defined in \eqref{prob_w_conv_reg}), we have
 \begin{align*}
x_{k}-\gamma_{k}\nabla F_{k}(x_{k})=\dfrac{1}{1+\lambda\mu_{k}}\left(\lambda\mu_{k}P_{\mathcal{B}}(x_{k})+\prox_{\mu_{k}\widetilde{g}}(x_{k})\right).
 \end{align*}
 Thus, the Algorithm~\ref{alg_w_convex_sev} for solving problem~\eqref{app_w_1} reduces to
 \begin{algorithm}[H]
\caption{Projected variable smoothing for \eqref{app_w_1}}\label{alg_max_disp_1}
\begin{algorithmic}
\Require $x_{1}\in \ker R$, $\alpha\in\left]0,1\right[$, and $C>0$ such that $4C\leq 1$.
\For{$k=1,2,\ldots$}
\State $\mu_{k}\gets Ck^{-\alpha}$ and  $x_{k+1}\gets P_{\ker R}\left(\dfrac{1}{1+\lambda\mu_{k}}(\lambda\mu_{k}P_{\mathcal{B}}(x_{k})+\prox_{\mu_{k}\widetilde{g}}(x_{k})\right))$.
\EndFor
\end{algorithmic}
\end{algorithm}
The proximity operator of $\widetilde{g}$ in \eqref{fun_g_til_1} can be calculated by using Proposition~\ref{p_prox_sup_w_aff}, while the projection onto $\mathcal{B}=B(0,r)$ is well-known.  On the other hand, in the context of problem~\eqref{app_w_2}, $\nabla H(\mathbf{x})=(\lambda(x_{1}-P_{\mathcal{B}}(x_{1})),0,\ldots,0)$, which implies that $\nabla H$ is Lipschitz with constant $L_{\nabla H}=\lambda$. Then, as in the case of problem~\eqref{app_w_1}, $\gamma_{k}=\frac{\mu_{k}}{1+\lambda\mu_{k}}$. Hence, for all $k\geq 1$,
\begin{align*}
\mathbf{x}^{k}-\gamma_{k}\nabla F_{k}(\mathbf{x}^{k})=\frac{1}{1+\lambda\mu_{k}}\left(\mu_{k}(\lambda \mathbf{x}^{k}-\nabla H(\mathbf{x}^{k}))+\prox_{\mu_{k}g}(\mathbf{x}^{k})\right).
\end{align*}
Therefore, the Algorithm~\ref{alg_w_convex_sev} for solving problem~\eqref{app_w_2} reduces to
 \begin{algorithm}[H]
\caption{Projected variable smoothing for \eqref{app_w_2}}\label{alg_max_disp_2}
\begin{algorithmic}
\Require $\mathbf{x}^{1}\in V:=(\ker R)^{N}\cap\mathcal{D}$, $\alpha\in\left]0,1\right[$, and $C>0$ such that $4C\leq 1$.
\For{$k=1,2,\ldots$}
\State Set $\mu_{k}\gets Ck^{-\alpha}$ and $\mathbf{z}^{k}\gets (\lambda(x_{1}^{k}-P_{\mathcal{B}}(x_{1}^{k})),0,\ldots,0)$
\State $\mathbf{x}^{k+1}\gets P_{V}(\dfrac{1}{1+\lambda\mu_{k}}\left(\mu_{k}(\lambda \mathbf{x}^{k}-\mathbf{z}^{k})+\prox_{\mu_{k}g}(\mathbf{x}^{k})\right))$.
\EndFor
\end{algorithmic}
\end{algorithm}
The proximity operator of $g$ in \eqref{fun_g_1} has a closed-form and is given by Propositions~\ref{p_prox_sup_w_quad} and \ref{prop_prob_aux_w}. For the numerical results, we consider $n=3$, $N=10$, $\alpha=1/3$, $C=1/4$, $r=1$, and $R=\begin{pmatrix} 1& 1 & 1
\end{pmatrix}$. That is, $\ker R=\{(x,y,z)\in\R^{3}:x+y+z=0\}$. We choose to stop every algorithm when the norm of the difference between two consecutive iterations is less than $10^{-5}$. On the other hand, for every $i\in\{1,\ldots,N\}$, we consider vectors $u_{i}$ of the form $2\texttt{rand}(n,1)$, where $\texttt{rand}(n,1)$ stands a random vector with components in $\left[0,1\right]$. Since $\mathcal{B}=B(0,r)$, then the objective function in \eqref{app_w_1} and \eqref{app_w_2} is
\begin{equation*}
F_{\lambda}(x)=\frac{\lambda}{2}(\max\{\|x\|-r,0\})^{2}+\max_{1\leq i\leq N}-\|x-u_{i}\|^{2}.
\end{equation*}
In Figure~\ref{fig_1-2}, we illustrate the value of the objective function $F_{\lambda}$ with respect to the time (in seconds) for the algorithms \ref{alg_max_disp_1} and \ref{alg_max_disp_2}. 
\begin{figure}[htbp]
\begin{center}
\includegraphics[scale=1.06]{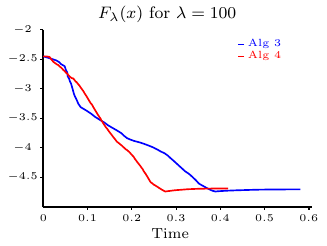}
\includegraphics[scale=1.06]{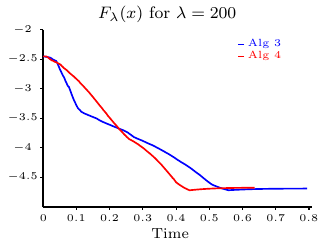}
\end{center}
\caption{Performance of algorithms  \ref{alg_max_disp_1} and \ref{alg_max_disp_2} for  $\lambda=100$ (left) and $\lambda=200$  (right).}
\label{fig_1-2}
\end{figure}

We observe that for both cases of $\lambda$, the Algorithm~\ref{alg_max_disp_2} converges in less time than Algorithm~\ref{alg_max_disp_1}. This is because Algorithm~\ref{alg_max_disp_1} requires computing the proximity operator of $\widetilde{g}$, which takes more time than calculating the proximity of $g$. However, the values of the objective function in the last iteration in algorithms \ref{alg_max_disp_1} and \ref{alg_max_disp_2} are $-4.7043$ and $-4.6885$ respectively for $\lambda=100$, and of $-4.6901$ and $-4.6763$ respectively for $\lambda=200$. Therefore, the Algorithm~\ref{alg_max_disp_1} provides a better approximation to the solution of the max dispersion problem.
 \section{Conclusions}
 \label{sec_conc}
 In this paper, we consider the problem of minimizing the sum of a smooth function and the composition of a weakly convex function with a bounded linear operator over a closed vector subspace, generalizing the framework studied in \cite{bohm2021variable}, which considers the case where the vector subspace is the entire space. We propose a variable smoothing algorithm that includes projection onto the vector subspace and we prove a complexity of $\mathcal{O}(\epsilon^{-3})$ to achieve an $\epsilon$-approximate solution. In both the literature and the proposed algorithm, it is typically assumed that the weakly convex function is Lipschitz in order to obtain an approximate solution. However, we show that if the sequence generated by the proposed algorithm is bounded, this assumption is not necessary.

 The proposed algorithm requires computing the proximity operator of a weakly convex function. In this context, we compute the proximity operator of a supremum of weakly convex functions in two important cases. In one case, the proximity operator has a closed formula,  and in the other, we provide an algorithm that converges to the proximity operator. The supremum function appears in the so-called distributionally robust optimization problem \cite{sun2021robust,MR4361979}. Another problem where our algorithm can be applied, which also involves a supremum function, is the max-dispersion problem (or max-min location problem) over the intersection of a  subspace and a compact convex set.
 
 \begin{acknowledgements}
The authors were supported by ANID Chile under grants Fondecyt Regular  N$^{\circ}$ 1240120 (P. P\'erez-Aros and E. Vilches), Fondecyt Regular N$^{\circ}$ 1220886 (P. P\'erez-Aros and E. Vilches), Fondecyt Regular N$^{\circ}$ 1240335 (P. P\'erez-Aros),  Proyecto de Exploraci\'on 13220097 (P. P\'erez-Aros and E. Vilches),  CMM BASAL funds for Center of Excellence FB210005 (P. P\'erez-Aros and E. Vilches), Project ECOS230027 (P. P\'erez-Aros and E. Vilches), MATH-AMSUD 23-MATH-17 (P. P\'erez-Aros and E. Vilches). Beca de Doctorado Nacional 21210951 (S. L\'opez-Rivera).
\end{acknowledgements}

\appendix

\section{Appendix}
\vspace{-3mm}
\subsection{Penalty Method}\label{penalty_method}
\vspace{-3mm}
This section is based in \cite[Section~21.1]{luenberger1984linear}. Consider the following optimization problem
  \begin{equation}
  \label{const_prob}
    \min_{x\in V\cap \mathcal{B}}f(x),
  \end{equation}
  where $f\colon\R^{n}\rightarrow\R$ is lower semicontinuous, $V$ is a closed vector subspace of $\R^{n}$, and $\mathcal{B}$ is a compact convex subset of $\R^{n}$. Note that \eqref{const_prob} has solutions since $f$ is lower semicontinuous and $V\cap \mathcal{B}$ is compact. The idea of the penalty function method is replace problem~\eqref{const_prob} by a problem of the form $\min_{x\in V}f(x)+\lambda P(x)$, where $\lambda>0$ and $P\colon\R^{n}\rightarrow\R$ is a lower semicontinuous function satisfying that $P(x)\geq 0$ for all $x\in\R^{n}$ and $P(x)=0$ if and only if $x\in \mathcal{B}$. In order to solve problem~\eqref{const_prob} by the penalty function method, let $\lambda_{k}\to +\infty$ be a sequence such that $\lambda_{k}>0$ and $\lambda_{k+1}>\lambda_{k}$ for all $k\in\mathbb{N}$. Consider $q(\lambda,x):=f(x)+\lambda P(x)$ and assume that for all $k\in\mathbb{N}$ the problem $\min_{x\in V} q(\lambda_{k},x)$ has a solution $x_{k}\in V$.
  \begin{lemma}
\label{lema_penalty}
For every $k\in\mathbb{N}$, we have
$$
q(\lambda_{k},x_{k})\leq q(\lambda_{k+1},x_{k+1}),\,  P(x_{k})\geq P(x_{k+1}) \textrm{ and } f(x_{k})\leq f(x_{k+1}).
$$
  \end{lemma}
  \begin{proof}
Since $\lambda_k$ is increasing, we have
\begin{align*}
q(\lambda_{k+1},x_{k+1})\geq f(x_{k+1})+\lambda_{k}P(x_{k+1})\geq f(x_{k})+\lambda_{k}P(x_{k})=q(\lambda_{k},x_{k}),
\end{align*}
thereby proving the first inequality. Next, observe that 
$$
f(x_{k})+\lambda_{k}P(x_{k})\leq f(x_{k+1})+\lambda_{k}P(x_{k+1}) \leq f(x_k)+\lambda_{k+1}P(x_k)+(\lambda_k-\lambda_{k+1})P(x_{k+1}),
$$
which implies the second inequality. Finally, by using the second inequality, we get that 
$$
f(x_{k})+\lambda_{k}P(x_{k})\leq f(x_{k+1})+\lambda_{k}P(x_{k+1})\leq f(x_{k+1})+\lambda_{k}P(x_{k}),$$ which implies the last inequality. \qed
  \end{proof}
\begin{lemma}
\label{lema_penalty_2} If $x^{*}$ solves problem~\eqref{const_prob}, then  $f(x^{*})\geq q(\lambda_{k},x_{k})\geq f(x_{k})$ for all $k\in \mathbb{N}$.
\end{lemma}
\begin{proof} $f(x^{*})=f(x^{*})+\lambda_{k}P(x^{*})\geq f(x_{k})+\lambda_{k}P(x_{k})\geq f(x_{k})$.  \qed
\end{proof}
\begin{theorem}
A cluster point of a sequence generated by the penalty method solves \eqref{const_prob}.
\end{theorem}
\begin{proof}
Let $(x_{k})$ generated by the penalty method and $\bar{x}$ a cluster point. Let $f^{*}$ the value of \eqref{const_prob}. From Lemmas \ref{lema_penalty} and \ref{lema_penalty_2}, it follows that $(f(x_{k}))$ is increasing and bounded above by $f^{*}$. Hence, $(f(x_{k}))$ converges and $f(\overline{x})\leq \lim f(x_{k})$.  Similarly, $q^{*}:=\lim q(\lambda_{k},x_{k})$ and $\lim P(x_k)$ exist. Therefore,  $0\leq \lim\lambda_{k}P(x_{k})=\lim q(\lambda_{k},x_{k})-\lim f(x_{k})\leq q^{*}-f(\overline{x})$, which, since $\lambda_{k}\rightarrow\infty$ and $P(x_{k})\geq 0$, implies that $\lim P(x_k)=0$. Moreover, $0\leq P(\bar{x})\leq \lim P(x_k)=0$, thereby proving that $\bar{x}\in \mathcal{B}$. Finally, from Lemma~\ref{lema_penalty_2}, we conclude that $f(\overline{x})\leq\lim_{k\rightarrow\infty}f(x_{k})\leq f^{*}$, which ends the proof. \qed
\end{proof}
\vspace{-3mm}
\subsection{ Key inequalities}
\vspace{-3mm}
\begin{lemma}\label{claim_des}
Assume that $\alpha \in ]0,1[$. Then, for all $K\in \mathbb{N}$
$$
(K+1)^{1-\alpha}-1\geq (2^{1-\alpha}-1)K^{1-\alpha}.
$$
\end{lemma}
\begin{proof}
Let us define $\varphi(x)=(x+1)^{1-\alpha}-1- \theta x^{1-\alpha}$ for all $x\geq 1$, where $\theta=2^{1-\alpha}-1$. Since $\alpha\in\left]0,1\right[$, $0<\theta<1$. Note that $\varphi(1)=0$. Let $x\geq 1$. Then, $\varphi'(x)=(1-\alpha)(x+1)^{-\alpha}-(1-\alpha)\theta x^{-\alpha}$. We will see that $\varphi'(x)\geq 0$.  Now, we observe that $\theta<2^{-\alpha}$ if and only if $\frac{\theta^{1/\alpha}}{1-\theta^{1/\alpha}}<1$. Therefore, $x\geq 1>\frac{\theta^{1/\alpha}}{1-\theta^{1/\alpha}}$. Now, we have that $x\geq \frac{\theta^{1/\alpha}}{1-\theta^{1/\alpha}}$ if and only if  $(x+1)^{-\alpha}\geq \theta x^{-\alpha}$.  Whence $\varphi'(x)\geq 0$. It follows that $\varphi(x)\geq \varphi(1)=0$ for all $x\geq 1$, which implies the result. \qed
\end{proof}

\begin{lemma}\label{claim_N}
Let $\alpha\in ]0,1[$ and $N\in \mathbb{N}$ such that $\theta:=(1+1/N)^{1-\alpha}-1<1/2$. Then, for all $K\geq N$
$$
(K+1)^{1-\alpha}-N^{1-\alpha}\geq \theta K^{1-\alpha}.
$$
\end{lemma}
\begin{proof}
Define  $\varphi(x):=(x+1)^{1-\alpha}-N^{1-\alpha}-\theta x^{1-\alpha}$. By definition of $\theta$, we deduce that $\varphi(N)=0$. Let $x\geq N$. Then, $\varphi'(x)=(1-\alpha)(x+1)^{-\alpha}-(1-\alpha)\theta x^{-\alpha}$. 
Since $\alpha\in\left]0,1\right[$, we have that $(2^{1/\alpha}-1)N>1$, which is equivalent to $1/2<(\frac{N}{N+1})^{\alpha}$. Hence, $\theta<1/2<(\frac{N}{N+1})^{\alpha}$. Then, $\theta^{1/\alpha}\leq \frac{N}{N+1}$, which implies that $x\geq N\geq \frac{\theta^{1/\alpha}}{1-\theta^{1/\alpha}}$. Note that $x\geq \frac{\theta^{1/\alpha}}{1-\theta^{1/\alpha}}$ if and only if  $(x+1)^{-\alpha}\geq \theta x^{-\alpha}$. Hence $\varphi'(x)\geq 0$ and $\varphi(x)\geq \varphi(N)=0$ for all $x\geq N$.  \qed
\end{proof}

\begin{proposition}
\label{prop_1.1}
Let $g\colon \H\rightarrow \R\cup\{+\infty\}$ be a proper, l.s.c., and $\rho$-weakly convex function. 
 Then there exists $\mu_{0}<\min\{1,\frac{1}{2\rho}\}$ such that for every 
 $\mu\in\left]0,\mu_{0}\right[$, $x\in\H$, and $x_{0}\in\dom g$ we have
$$
\|\prox_{\mu g}(x)\|^{2}\leq 4(M+5\|x\|^{2}),
$$
where $M=\max\{0,g(x_{0})\}+2\|x_{0}\|^{2}+1/2$.
	\end{proposition}
	\begin{proof}
    We claim that there exists $\theta>0$ such that for all $x\in\H$,
\begin{align}
\label{claim_1_w_convex}
g(x)\geq -\dfrac{\theta}{2}(1+\|x\|^{2}).
\end{align}
Since $g+\dfrac{\rho}{2}\|\cdot\|^{2}$  has an affine minorant, there exist $x^{*}\in\H$ and $\beta\in\R$ such that 
\begin{align*}
g(x)+\dfrac{\rho}{2}\|x\|^{2}\geq \langle x^{*},x\rangle+\beta \geq -\frac{\|x^{*}\|^{2}}{2}-\frac{\|x\|^{2}}{2}+\beta \textrm{ for all } x\in \H.
\end{align*}
Then, $g(x)\geq -\frac{1}{2}(\|x^{*}\|^{2}-2\beta)-\frac{1}{2}(1+\rho)\|x\|^{2}\geq -\frac{1}{2}\max\{\|x^{*}\|^{2}-2\beta,1+\rho\}(1+\|x\|^{2})$. Therefore, $g$ satisfies \eqref{claim_1_w_convex} with $\theta=\max\{\|x^{*}\|^{2}-2\beta,1+\rho\}>0$. Let us define $\mu_{0}=\min\{\frac{1}{2\theta},\frac{1}{1+\theta}\}$.

Let $\mu\in ]0,\mu_0 [$, let $x\in\H$, let $y=\prox_{\mu g}(x)$, and$x_{0}\in\dom g\neq\emptyset$. By \eqref{claim_1_w_convex}, we have that
		\begin{align*}
			-\frac{\theta}{2}-\frac{\theta}{2}\|y\|^{2}+\frac{1}{2\mu}\|y-x\|^{2}&\leq g(y)+\frac{1}{2\mu}\|y-x\|^{2}
			\leq g(x_{0})+\frac{1}{2\mu}\|x_{0}-x\|^{2}.
		\end{align*}
		Then,
		\begin{align*}
			-\frac{\theta}{2}-\frac{\theta}{2}\|y\|^{2}+\frac{1}{2\mu}\|y\|^{2}-\frac{1}{\mu}\langle y,x \rangle\leq g(x_{0})+\frac{1}{2\mu}\|x_{0}\|^{2}-\frac{1}{\mu}\langle x_{0},x\rangle
			\leq g(x_{0})+\frac{1}{\mu}\|x_{0}\|^{2}+\frac{1}{2\mu}\|x\|^{2}.
		\end{align*}
		Now, from $ab\leq \frac{\alpha}{2}a^{2}+\frac{1}{2\alpha}b^{2}$ with $a=\|y\|$, $b=\|x\|$ and $\alpha=\frac{1-\mu\theta}{2}>0$, we have that
		\begin{align*}
			-\frac{\theta}{2}-\frac{\theta}{2}\|y\|^{2}+\frac{1}{2\mu}\|y\|^{2}\leq \frac{1}{2\mu}\left(\frac{1-\mu\theta}{2}\right)\|y\|^{2}+g(x_{0})+\frac{1}{\mu}\|x_{0}\|^{2}
			+\frac{1}{\mu(1-\mu\theta)}\|x\|^{2}+\frac{1}{2\mu}\|x\|^{2}.
		\end{align*}
		Rearranging, we obtain 
		\begin{align*}
			\frac{1}{4}\left(\frac{1-\mu\theta}{\mu}\right)\|y\|^{2}\leq g(x_{0})+\frac{1}{\mu}\|x_{0}\|^{2}+\left(\frac{1}{\mu(1-\mu\theta)}+\frac{1}{2\mu}\right)\|x\|^{2}+\frac{\theta}{2},
		\end{align*}
		which implies that
		\begin{align*}
			\|y\|^{2}\leq 4\left[ \frac{\mu}{1-\mu\theta}g(x_{0})+\frac{1}{1-\mu\theta}\|x_{0}\|^{2}+\left(\frac{1}{(1-\mu\theta)^{2}}+\frac{1}{2(1-\mu\theta)}\right)\|x\|^{2}+\frac{\mu\theta}{2(1-\mu\theta)}\right].
		\end{align*}
		Note that $\frac{\mu}{1-\mu\theta}<1$ (since $\mu<\frac{1}{1+\theta}$), $\frac{1}{1-\mu\theta}<2$, and $\frac{\mu\theta}{2(1-\mu\theta)}<\frac{1}{2}$ (since $\mu<\frac{1}{2\theta}$). Thus,  $\|y\|^{2}\leq 4(\max\{0,g(x_{0})\}+2\|x_{0}\|^{2}+5\|x\|^{2}+1/2)$. \qed
	\end{proof}
\begin{proposition}
\label{prop_prox_w_convex}
Let $S$ be a nonempty closed convex and bounded subset of $\R^{m}$ and let $g\colon\R^{m}\rightarrow\R$ be a $\rho$-weakly convex and lower semicontinuous function. Then, there exist $\ell\geq 0$ and $\mu_{0}\in\left]0,1\right[$ (with $\mu_{0}$ depending on $\rho$) such that for all $\mu\in\left]0,\mu_{0}\right[$ and $x\in S$,
\begin{align*}
\|\nabla g_{\mu}(x)\|=\dfrac{1}{\mu}\|x-\prox_{\mu g}(x)\|\leq \ell.
\end{align*}
\end{proposition}
\begin{proof}
By Proposition~\ref{prop_1.1}, there exist $\widetilde{\mu}_{0}\in]0,1[$ and $M\geq 0$ such that for all $\mu\in\left]0,\widetilde{\mu}_{0}\right[$,
\begin{align*}
\|\prox_{\mu g}(x)\|^{2}\leq 4(M+5\|x\|^{2}) \textrm{ for all } x\in \mathbb{R}^m.
\end{align*}
Let us define $\widetilde{S}:=\{x\in\R^{m}:d(x,S)\leq 1\}$. Since $\widetilde{S}$ is nonempty closed, convex, and bounded, there exists $\widetilde{\ell}\geq 0$ such that $g+\frac{\rho}{2}\|\cdot\|^{2}$ is $\widetilde{\ell}$-Lipschitz on $\widetilde{S}$ and there exists $\widetilde{M}\geq 0$ such that $\|x\|\leq \widetilde{M}$ for all $x\in \widetilde{S}$. Let $x,y\in\widetilde{S}$. Then,
\begin{align*}
|g(x)-g(y)|&\leq | g(x)-g(y)+\dfrac{\rho}{2}\|x\|^{2}-\dfrac{\rho}{2}\|y\|^{2}|+\dfrac{\rho}{2}|\|x\|^{2}-\|y\|^{2}|\\
&\leq \widetilde{\ell}\|x-y\|+\dfrac{\rho}{2}|\|x\|+\|y\||\cdot|\|x\|-\|y\||\\
&\leq \widetilde{\ell}\|x-y\|+\rho\widetilde{M}\|x-y\|\\
&=(\widetilde{\ell}+\rho\widetilde{M})\|x-y\|.
\end{align*}
Hence, $g$ is Lipschitz on $\widetilde{S}$. Similarly, $g$ is Lipschitz in $S$.
Now, for all $x\in S$ and $\mu\in ]0,\widetilde{\mu}_{0}[$,
\begin{align*}
\dfrac{1}{2\mu}\|x-\prox_{\mu g}(x)\|^{2}&\leq g(x)-g(\prox_{\mu g}(x))\\
&\leq g(x)+\frac{\rho}{2}\|\prox_{\mu g}(x)\|^{2}-\langle x^{*},\prox_{\mu g}(x)\rangle-\beta\\
&\leq g(x)+2\rho(M+5\|x\|^{2})+\|x^{*}\|\|x-\prox_{\mu g}(x)\|+\|x^{*}\|\|x\|-\beta\\
&\leq g(x)+\mu\|x^{*}\|^{2}+\frac{1}{4\mu}\|x-\prox_{\mu g}(x)\|^{2}+2\rho(M+5\|x\|^{2})+\|x^{*}\|\|x\|-\beta,
\end{align*}
since $ab\leq \frac{c^{2}}{2}a^{2}+\frac{1}{2c^{2}}b^{2}$ for $c=\sqrt{2\mu}$. Hence, for all $x\in S$ and $\mu\in\left]0,\widetilde{\mu}_{0}\right[$, we have
\begin{align*}
\|x-\prox_{\mu g}(x)\|^{2}\leq 4\mu(g(x)+\mu\|x^{*}\|^{2}+2\rho(M+5\|x\|^{2})+\|x^{*}\|\|x\|-\beta).
\end{align*}
Since $g$ is Lipschitz in $S$ and $S$ is bounded, the right-hand side of the latter inequality is uniformly bounded in $S$. Thus, we can to find $\widetilde{m}>0$ such that for all $x\in S$ and $\mu\in\left]0,\widetilde{\mu}_{0}\right[$
\begin{align*}
\|x-\prox_{\mu g}(x)\|\leq \sqrt{\mu}\widetilde{m}.
\end{align*}
Set $\mu_{0}=\min\{\widetilde{\mu}_{0},\frac{1}{\widetilde{m}^{2}}\}\in]0,1[$. Then, for all $x\in S$ and $\mu\in\left]0,\mu_{0}\right[$, $\|x-\prox_{\mu g}(x)\|\leq 1$, whence $\prox_{\mu g}(x)\in \widetilde{S}$.  Finally, the result follows from the formula $\nabla g_{\mu}(x)\in\partial g(\prox_{\mu g}(x))$ for all $x\in\R^{m}$ and the $\ell$-Lipschitz continuity of $g$ over in $\widetilde{S}$. \qed
\end{proof}

\subsection{Proof of Proposition \ref{prop_prob_aux_w}}\label{Annex-C}
\vspace{-3mm}
\emph{Proof of Proposition \ref{prop_prob_aux_w}}:  The set defined in \eqref{def_k_w} is nonempty since $N-1$ is in that set. Hence $k$ is well defined. Note that $\nabla \phi(p)=\left(\frac{-\alpha_{i}}{(2\mu p_{i}-1)^{2}}\right)_{i=1}^{N}$ for all $p\in\R^{N}$. Then, since $-\phi$ is convex on $\Delta_{N}$, by the definition of $\Delta_{N}$ and the KKT's conditions, it follows that it is enough to prove that there exists $\tau\in \R$ and $(\eta_{i})_{i=1}^{N}\in\R_{+}^{N}$ such that $\sum_{i=1}^{N}p_{i}=1$ and 
		\begin{equation}
			\label{cond_kkt_w}
			\alpha_{i}(2\mu p_{i}-1)^{-2}+\tau-\eta_{i}=0, \eta_{i}p_{i}=0,  p_{i}\geq 0 \textrm{ for all } i\in\{1,\ldots,N\},
		\end{equation}
		where $p\in\R^{N}$ is defined by \eqref{sol_prob_aux_w}. Consider $\tau:=\frac{-1}{(N-k-2\mu)^{2}}(\sum_{j\notin I_{k}}\sqrt{\alpha_{j}})^{2}\in\R$ and $(\eta_{i})_{i=1}^{N}\in \R^{N}$ defined by $\eta_{i}=\tau+\alpha_{i}$ if $i\in I_{k}$ and $\eta_{i}=0$ if $i\notin I_{k}$. Thus, we have the second condition in \eqref{cond_kkt_w}. Let us now prove the first condition in \eqref{cond_kkt_w}. Let $i\in\{1,\ldots,N\}$. If $i\in I_{k}$, then $p_{i}=0$ and $\frac{\alpha_{i}}{(2\mu p_{i}-1)^{2}}+\tau-\eta_{i}=\alpha_{i}+\tau-\eta_{i}=0$. If $i\notin I_{k}$, then by \eqref{sol_prob_aux_w},  $(1-2\mu p_{i})^{2}=(N-k-2\mu)^{2}\alpha_{i}(\sum_{j\notin I_{k}}\sqrt{\alpha_{j}})^{-2}$ and hence
		\begin{align*}
			\alpha_{i}(1-2\mu p_{i})^{-2}+\tau-\eta_{i}=(N-k-2\mu)^{-2}(\sum_{j\notin I_{k}}\sqrt{\alpha_{j}})^{2}+\tau-\eta_{i}=-\eta_{i}=0,
		\end{align*}
		which proves the first condition in \eqref{cond_kkt_w}.
		We claim that $\eta_{i}\geq 0$ for all $i\in I_{k}$ (note that if $k=0$, the latter is direct since $I_{0}=\emptyset$, so in order to prove this claim we assume that $k>0$). Let $i\in I_{k}$. Then $\alpha_{i}\geq \alpha_{\ell_{k}}$. Now, by definition of $k$, we have that $k-1$ is not in the set in \eqref{def_k_w}, that is,
		\begin{align*}
			(N-k+1-2\mu)\sqrt{\alpha_{\ell_{k}}}\geq \sum_{j\notin I_{k-1}}\sqrt{\alpha_{j}} = \sqrt{\alpha_{\ell_{k}}}+\cdots+\sqrt{\alpha_{\ell_{N}}},
		\end{align*}
		which yields that $(N-k-2\mu)\sqrt{\alpha_{i}}\geq (N-k-2\mu)\sqrt{\alpha_{\ell_{k}}}\geq \sqrt{\alpha_{\ell_{k+1}}}+\cdots+\sqrt{\alpha_{\ell_{N}}}=\sum_{j\notin I_{k}}\sqrt{\alpha_{j}}$.  Hence, $\eta_{i}=\tau+\alpha_{i}=\alpha_{i}-\frac{1}{(N-k-2\mu)^{2}}(\sum_{j\notin I_{k}}\sqrt{\alpha_{j}})^{2}\geq 0$. We claim now that $p_{i}\geq 0$ for all $i\notin I_{k}$. Let $i\notin I_{k}$. Then $\alpha_{i}\leq \alpha_{\ell_{k+1}}$ and therefore
		\begin{equation*}
			\frac{(N-k-2\mu)\sqrt{\alpha_{i}}}{\sum_{j\notin I_{k}}\sqrt{\alpha_{j}}}\leq \frac{(N-k-2\mu)\sqrt{\alpha_{\ell_{k+1}}}}{\sum_{j\notin I_{k}}\sqrt{\alpha_{j}}}<1,
		\end{equation*}
		where the last inequality is by the definition of $k$. Thus, from \eqref{sol_prob_aux_w}, we obtain that $p_{i}\geq 0$. Finally, since $|I_{k}^{c}|=N-k$, we deduce from \eqref{sol_prob_aux_w}, that $\sum_{i=1}^{N}p_{i}=1$. \qed

\end{document}